\documentclass[preprint,number, 3p]{elsarticle}

\usepackage[utf8]{inputenc}
\usepackage[english]{babel}
\usepackage{etoolbox}
\usepackage{amsmath}
\usepackage{amsfonts}
\usepackage{amssymb}
\usepackage{amsthm}
\usepackage{enumitem}
\usepackage{relsize}
\usepackage{thmtools}
\usepackage{dsfont}
\usepackage{bm}
\usepackage{hyperref}

\makeatletter
	\def\ps@pprintTitle{%
	\let\@oddhead\@empty
	\let\@evenhead\@empty
	\def\@oddfoot{\centerline{\thepage}}%
	\let\@evenfoot\@oddfoot}
	
	\renewcommand\section{\@startsection {section}{1}{\z@}%
	           {-3.5ex \@plus -1ex \@minus -.2ex}%
	           {2.3ex \@plus.2ex}%
	           {\normalfont\large\bfseries}}
	\renewcommand\subsection{\@startsection{subsection}{2}{\z@}%
	           {-3.25ex\@plus -1ex \@minus -.2ex}%
	           {1.5ex \@plus .2ex}%
	           {\normalsize\bfseries\itshape}}
\makeatother

\allowdisplaybreaks

\newcommand{\eps}{\varepsilon}
\newcommand{\mb}[1]{\mathbf{#1}}
\newcommand{\warrow}{\overset{w}{\longrightarrow}}
\newcommand{\ccdot}{\kern-.12em\cdot\kern-.12em}
\newcommand\restrict[1]{\raisebox{-.5ex}{$|$}_{#1}}
\DeclareMathOperator*{\esssup}{ess\,sup}

\makeatletter
	\@ifdefinable\@latex@chi{\let\@latex@chi\chi}
	\renewcommand*\chi{{\@latex@chi\smash[t]{\mathstrut}}} 
\makeatother

\newtheoremstyle{slplain}
  {0.4cm}
  {0.4cm}
  {\upshape}
  {}
  {\bfseries}
  {.}
  { }
  {}
  
\newtheoremstyle{itplain}
    {0.4cm}
    {0.4cm}
    {\itshape}
    {}
    {\bfseries}
    {.}
    { }
    {}
  
\declaretheorem[style=slplain,numberwithin=section]{definition}
\declaretheorem[style=slplain,sibling=definition]{example}
\declaretheorem[style=slplain,sibling=definition]{remark}

\declaretheorem[style=itplain,sibling=definition]{theorem}

\declaretheorem[style=itplain,sibling=definition]{proposition}
\declaretheorem[style=itplain,sibling=definition]{lemma}
\declaretheorem[style=itplain,sibling=definition]{corollary}

\patchcmd{\MaketitleBox}{\footnotesize\itshape\elsaddress\par\vskip36pt}{\footnotesize\itshape\elsaddress\par\parbox[b][36pt]{\linewidth}{\vfill\hfill\textnormal{\today}\hfill\null\vfill}}{}{}%
\patchcmd{\pprintMaketitle}{\footnotesize\itshape\elsaddress\par\vskip36pt}{\footnotesize\itshape\elsaddress\par\parbox[b][36pt]{\linewidth}{\vfill\hfill\textnormal{\today}\hfill\null\vfill}}{}{}%

\begin{document}

\begin{frontmatter}

\title{\mbox{ }\\[0.6cm] Lévy processes with respect to the index Whittaker convolution\vspace{5pt}}

\author[myaddress1]{Rúben Sousa\corref{mycorrespondingauthor}}
\cortext[mycorrespondingauthor]{Corresponding author}
\ead{rubensousa@fc.up.pt}

\author[myaddress2]{Manuel Guerra}
\ead{mguerra@iseg.ulisboa.pt}

\author[myaddress1]{Semyon Yakubovich}
\ead{syakubov@fc.up.pt}

\address[myaddress1]{CMUP, Departamento de Matemática, Faculdade de Ciências, Universidade do Porto,\linebreak Rua do Campo Alegre 687, 4169-007 Porto, Portugal\vspace{2pt}}

\address[myaddress2]{CEMAPRE and ISEG (School of Economics and Management), Universidade de Lisboa,\linebreak Rua do Quelhas, 1200-781 Lisbon, Portugal}

\begin{abstract}
The index Whittaker convolution operator, recently introduced by the authors, gives rise to a convolution measure algebra having the property that the convolution of probability measures is a probability measure. In this paper, we introduce the class of Lévy processes with respect to the index Whittaker convolution and study their basic properties. We prove that the square root of the Shiryaev process belongs to our family of Lévy process, and this is shown to yield a martingale characterization of the Shiryaev process analogous to Lévy's characterization of Brownian motion.
	 
Our results demonstrate that a nice theory of Lévy processes with respect to generalized convolutions can be developed even if the usual compactness assumption on the support of the convolution fails, shedding light into the connection between the properties of the convolution algebra and the nature of the singularities of the associated differential operator.
\end{abstract}


\end{frontmatter}

\section{Introduction}

The \textit{index Whittaker transform} is the integral transform (of index type) defined by
\begin{equation} \label{eq:intro_Whit_transf}
(\mathrm{W}_{\!\alpha} g)(\tau) := \int_0^\infty g(x) W_{\alpha, i\tau}(x)\, x^{-2} dx, \qquad \tau \geq 0
\end{equation}
where $i$ is the imaginary unit, $\alpha < {1 \over 2}$ is a parameter and $W_{\alpha, \nu}(x)$ is the Whittaker function of the second kind. This transformation first appeared in \cite{wimp1964} as a particular case of an integral transform having the Meijer-G function in the kernel. Various results on the $L_p$ theory, such as the Plancherel theorem, were established in \cite{srivastava1998}. In its general form, the index Whittaker is connected with the Asian option pricing problem in mathematical finance \cite{linetsky2004}. Furthermore, it includes as a particular case the Kontorovich-Lebedev transform, which is one of the most well-known index transforms and has a wide range of applications in physics \cite{sousayakubovich2017,yakubovichluchko1994,yakubovich1996}. The index Whittaker transform is, up to a simple transformation (see \cite{sousayakubovich2017}), the Sturm-Liouville type integral transform associated with the eigenfunction expansion of the differential operator
\begin{equation} \label{eq:iw_intro_Lop}
\mathcal{L} = - {1 \over 4} \Bigl[x^2 {d^2 \over dx^2} + \bigl(x^{-1} + (3-4\alpha) x\bigr){d \over dx} \Bigr].
\end{equation}
The operator $\mathcal{L}$ is also the negative of the generator of the \emph{index Whittaker diffusion}, defined as the square root of the so-called Shiryaev process, which is the solution of a stochastic differential equation derived by Shiryaev \cite{shiryaev1961} in the context of quickest detection problems and has various applications in fields such as sequential analysis and mathematical finance \cite{donatimartin2001,peskir2006}.

Given an integral transformation for which a Plancherel theorem holds, it is natural to study whether the integral transformation can be used to define a convolution operator inducing a Banach algebra structure in the space of finite complex Borel measures \cite{litvinov1987}. If, additionally, the convolution of probability measures is always a probability measure, it becomes meaningful to ask if a reasonable theory of infinitely divisible probability distributions with respect to the convolution can be developed. One of the purposes of this paper is to show that the \emph{index Whittaker convolution} of (probability) measures, in the form introduced in Section \ref{chap:iw_convalgebra}, provides an affirmative answer to these questions.

Another important goal of this paper is to introduce and develop the concepts of convolution semigroups, Lévy processes and Gaussian processes with respect to the index Whittaker convolution. We show that these Lévy processes constitute a family of Feller processes whose $L_2$-generator can be explicitly described; we also establish a Lévy-Khintchine representation for this class of Lévy processes, as well as a general martingale characterization. It turns out that the index Whittaker diffusion is a Gaussian process with respect to the index Whittaker convolution, and one of our main results (Theorem \ref{thm:iw_moment_diffusion_martchar}) gives an explicit martingale characterization of the index Whittaker diffusion which is analogous to Lévy's characterization of Brownian motion. This result can be restated as a Lévy-type characterization of the Shiryaev process. Actually, a simple change of variables in the definition of the generalized convolution would lead to a convolution algebra whose class of Lévy processes includes the Shiryaev diffusion process; our choice of definition is a matter of analytical convenience.

Some of our proofs are based on the Laplace-type integral representation for the Whittaker function, which is established in Section \ref{chap:laplacerep}. This integral representation is new and of independent interest.

In a recent manuscript \cite{sousaetal2018}, the authors derived a product formula for the Whittaker function $W_{\alpha, \nu}(x)$ whose kernel does not depend on the second parameter $\nu$ and is given in closed form in terms of the parabolic cylinder function. In the same work, this product formula was used to define the index Whittaker convolution of complex-valued functions. Here, in order to address the probabilistic properties of this convolution, we use a modified form of the same product formula to extend the index Whittaker convolution to the space of all finite measures.



The concepts of infinitely divisible distributions and Lévy processes with respect to generalized convolutions are not new, as much work has been done in the framework of hypergroups \cite{bloomheyer1995, rentzschvoit2000} and other abstract algebraic structures (see \cite{borowieckaolszewska2015} for work in the context of Urbanik convolution algebras). A generalized convolution $*$ on the half-line $[0,\infty)$ is said to define a hypergroup structure if $\delta_x * \delta_y$ is a probability measure for all $x,y \geq 0$ (here $\delta_x$ is the Dirac measure at $x$) and if it satisfies a set of axioms: associativity, continuity, existence of identity element, existence of an involution, and the following compactness condition: \vspace{0.5\topsep}
\begin{enumerate}[leftmargin=1.25\leftmargini]
\item[\textbf{(H\textsubscript{C})}] For each $x, y \geq 0$, the support $\mathrm{supp}(\delta_x * \delta_y)$ is compact, and the mapping $(x,y) \mapsto \mathrm{supp}(\delta_x * \delta_y)$ is continuous
\end{enumerate} \vspace{-0.5\topsep}
(the axioms are given in full in \cite[Section 1.1]{bloomheyer1995}). The hypergroup axioms allow for the development of a substantial theory of harmonic analysis which includes various results on infinitely divisible distributions and the associated stochastic processes; in fact, many of our results on processes related to the index Whittaker convolution are analogous to known results for Sturm-Liouville hypergroups on the half-line. The index Whittaker convolution, however, is not covered by the theory of hypergroups because it violates the compactness condition (H\textsubscript{C}).

As in the case of Sturm-Liouville hypergroups on $[0,\infty)$, the index Whittaker convolution $\delta_x * \delta_y$ is defined so that $(x,y) \mapsto \int_{[0,\infty)\!} f(\xi) (\delta_x * \delta_y)(d\xi)$ is a solution of the Cauchy problem
\[
\mathcal{L}_x u(x,y) = \mathcal{L}_y u(x,y), \qquad u(x,0) = u(0,x) = f(x), \qquad {\partial u \over \partial y}(x,0) = {\partial u \over \partial x}(0,y) = 0
\]
where, in our framework, $\mathcal{L}_x$ and $\mathcal{L}_y$ denote the differential operator \eqref{eq:iw_intro_Lop} acting on the variable $x$ and $y$ respectively; in Sturm-Liouville hypergroups the operator is instead of the form $\mathfrak{L} = -{d^2 \over dx^2} - {A'(x) \over A(x)} {d \over dx}$, where $A$ is a positive function satisfying a set of assumptions (given in \cite[Section 2]{zeuner1992}) which make it natural to regard the operator $\mathfrak{L}$ as a perturbed Bessel operator (see \cite{chebli1995}). The crucial difference between the two settings is connected with the usual classification of second-order partial differential equations, cf.\ e.g.\ \cite{bitsadze1964}: while in the hypergroup case the equation $\mathfrak{L}_x u = \mathfrak{L}_y u$ is uniformly hyperbolic on $(0,\infty)^2$, this is no longer true when $\mathfrak{L}$ is replaced by the operator \eqref{eq:iw_intro_Lop}, because in this case the equation $\mathcal{L}_x u = \mathcal{L}_y u$ is hyperbolic inside $(0,\infty)^2$ but has a parabolic degeneracy along the boundary of the positive quadrant. Therefore, while the compactness of $\mathrm{supp}(\delta_x * \delta_y)$ in Sturm-Liouville hypergroups reflects the finite speed of propagation feature of hyperbolic equations \cite[Proposition 3.7]{zeuner1992}, the property $\mathrm{supp}(\delta_x * \delta_y) = [0,\infty)\,$ ($x, y > 0$) of the index Whittaker convolution is a natural consequence of the instantaneous propagation phenomenon for parabolic equations. As far as the authors are aware, Sturm-Liouville operators leading to Cauchy problems with initial condition on the boundary of parabolic degeneracy have never been treated in the hypergroup-related literature. Our work thus demonstrates that the uniform hyperbolicity condition is not indispensable for building the generalized class of Lévy processes associated with (the generator of) a given one-dimensional diffusion process.

The outline of the paper is as follows: Section \ref{chap:preliminaries} sets notation and states some facts about the kernel of the index Whittaker transform. In Section \ref{chap:laplacerep} we establish the Laplace-type integral representation for the Whittaker function. Section \ref{chap:iw_convalgebra} is devoted to the index Whittaker convolution measure algebra: we extend the index Whittaker transform to (probability) measures, construct the generalized translation and convolution operators and study their main properties. The infinite divisibility of probability measures with respect to this convolution is investigated in Section \ref{chap:iw_infdiv}, where we establish a Lévy-Khintchine representation in which the index Whittaker transform plays a role parallel to that of the ordinary characteristic function in the classical Lévy-Khintchine formula. Finally, Lévy and Gaussian processes with respect to the index Whittaker convolution are addressed in Section \ref{chap:iw_levyprocess}, which contains the main results of this work.

\section{Preliminaries} \label{chap:preliminaries}

The following notation will be used throughout this article: $\mathrm{C}^k(I)$ denotes the space of $k$ times continuously differentiable functions on an interval $I$; $\mathrm{C}_\mathrm{c}^k(I)$ is its subspace of compactly supported functions; the spaces of bounded continuous functions and of continuous functions vanishing at infinity are denoted by  $\mathrm{C}_\mathrm{b}(I)$ and $\mathrm{C}_0(I)$, respectively. For a weight function $w$ defined on a set $E$, $L_p(E; w(x)\, dx)$ denotes the weighted $L_p$-space with norm
\[
\|f\|_{\scalebox{0.65}{$L_p(E; w(x)\, dx)\!$}} = \biggl( \int_E |f(x)|^p w(x)dx \biggr)^{\! 1/p} \quad (1 \leq p < \infty), \qquad\quad \|f\|_{\scalebox{0.65}{$L_\infty(E; w(x)\, dx)$}} = \esssup_{x \in E} |f(x)|.
\]
The space of probability (respectively, finite positive, finite complex) Borel measures on an interval $I$ will be denoted by $\mathcal{P}(I)$ (respectively, $\mathcal{M}_{b}(I)$, $\mathcal{M}_{\mathbb{C}}(I)$).

We define $\bm{W}_{\!\alpha,\nu}(x)$ as the following function of confluent hypergeometric type:
\begin{equation} \label{eq:bW_def}
\bm{W}_{\!\alpha,\nu}(x) := 2^\alpha x^{2\alpha} e^{1 \over 4x^2} W_{\alpha,\nu}\Bigl({1 \over 2x^2}\Bigr) = \bigl(2x^2\bigr)^{-{1 \over 2} + \alpha + \nu} \Psi\Bigl({1 \over 2} - \alpha - \nu, 1-2\nu; {1 \over 2x^2}\Bigr) \qquad (x > 0)
\end{equation}
where $W_{\alpha,\nu}(z)$ is the Whittaker function of the second kind \cite[\S13.14]{dlmf}, $\Psi(a,b;z)$ is the confluent hypergeometric function of the second kind or Tricomi function \cite[Chapter VI]{erdelyiI1953}, and $\alpha \in (-\infty,{1 \over 2})$, $\nu \in \mathbb{C}$ are parameters. Unless stated otherwise, the parameter $\alpha < {1 \over 2}$ is held fixed throughout the discussion.

By transformation of the Whittaker differential equation (see \cite[\S13.14(i)]{dlmf}), the function $\bm{W}_{\!\alpha,\nu}(x)$ is a solution of the differential equation $\mathcal{L} w = \bigl((\tfrac{1}{2}-\alpha)^2 - \nu^2\bigr) w$, where $\mathcal{L}$ is the operator \eqref{eq:iw_intro_Lop}, i.e.
\begin{equation} \label{eq:iw_Lop}
\mathcal{L} := -{1 \over 4}\biggl(x^2 {d^2 \over dx^2} + {[x^2 \mathrm{m}(x)]' \over \mathrm{m}(x)} {d \over dx} \biggr)
\end{equation}
being $\mathrm{m}$ the function
\begin{equation} \label{eq:iw_mweight}
\mathrm{m}(x) := x^{1-4\alpha} e^{-{1 \over 2x^2}\!}.
\end{equation}
It follows from \cite[Equation 13.19.3]{dlmf} that the limiting behavior of its derivatives as $x \to 0$ is given by
\begin{equation} \label{eq:bW_limderivatives}
\begin{aligned}
{d^{2n} \over dx^{2n}} \bm{W}_{\alpha,\nu}(x) & \xrightarrow[x \to 0\,]{} (-1)^n 2^{3n} \pi^{-{1 \over 2}} \Gamma(\tfrac{1}{2} + n) (\tfrac{1}{2} - \alpha + \nu)_n (\tfrac{1}{2} - \alpha - \nu)_n \\
{d^{2n+1} \over dx^{2n+1}} \bm{W}_{\alpha,\nu}(x) & \xrightarrow[x \to 0\,]{} 0
\end{aligned} \qquad (n \in \mathbb{N}_0)
\end{equation}
where $\Gamma(\cdot)$ is the Gamma function \cite[Chapter I]{erdelyiI1953} and $(a)_n$ is the Pochhammer symbol, $(a)_0 = 1$ and $(a)_n = \prod_{j=0}^{n-1} (a+j)$ for $n \in \mathbb{N}$. In particular, the function $\bm{W}_{\!\alpha,\nu}(x)$ extends continuously to $x = 0$ by setting $\bm{W}_{\!\alpha,\nu}(0) \equiv 1$. In addition, we have \cite[Equation 13.18.2]{dlmf}
\begin{equation} \label{eq:bW_equal1_param}
\bm{W}_{\!\alpha,{1 \over 2} - \alpha}(x) = 1 \qquad \text{for all } x > 0.
\end{equation}
Concerning the limiting behavior as $x \to \infty$, it is given by
\begin{equation} \label{eq:bW_liminfty}
\begin{aligned}
\bm{W}_{\alpha,\pm \nu}(x) & \sim {2^{-{1 \over 2} + \alpha + \nu\,} \Gamma(2\nu) \over \Gamma({1 \over 2} - \alpha + \nu)} x^{-1+2(\alpha + \nu)}, & \mathrm{Re}\,\nu > 0, \\
\bm{W}_{\alpha,\pm \nu}(x) & = O\bigl( x^{-1+2(\alpha + \mathrm{Re}\,\nu)} \bigr), & \mathrm{Re}\,\nu \geq 0, \nu \neq 0, \\
\bm{W}_{\alpha,0}(x) & = O\bigl( x^{-1+2\alpha}\log x \bigr).
\end{aligned}
\end{equation}
as can be seen using \cite[\S13.14(iii)]{dlmf}. In particular, $\bm{W}_{\!\alpha,\nu}(x) \xrightarrow[\,x \to \infty\,]{} 0 \;$ for each $\nu$ in the strip $|\mathrm{Re}\,\nu| < {1 \over 2} - \alpha$. We shall also make use of the following asymptotic expansion \cite[Theorem 1.11]{yakubovich1996}, which holds uniformly in $x \in [\eps, \infty)$:
\begin{equation} \label{eq:bW_itau_asymp}
\bm{W}_{\!\alpha,i\tau}(x) = 2^\alpha x^{2\alpha-1} \tau^{\alpha - {1 \over 2}} \exp\Bigl(\frac{1}{4x^2}-\frac{\pi\tau}{2}\Bigr) \cos\bigl[\tau\log(8\tau x^2) - \tau - \tfrac{\pi}{2}(\tfrac{1}{2} - \alpha) \bigr] \Bigl(1+O(\tau^{-1})\Bigr), \quad\; \tau \to +\infty.
\end{equation}

\section{A Laplace-type integral representation for the Whittaker function} \label{chap:laplacerep}

The following theorem gives an integral representation for the confluent hypergeometric-type function \eqref{eq:bW_def} which is, to the best of our knowledge, new; in particular, it cannot be found in standard references such as \cite{prudnikovI1986,prudnikovII1986,prudnikovIII1989}. We will call it the \emph{Laplace-type representation} for $\bm{W}_{\!\alpha, \nu}(x)$ because it is of the same form as the Laplace representation for the characters of Sturm-Liouville hypergroups, cf.\ \cite[(4.7)--(4.8)]{zeuner1992}.

\begin{theorem} \label{thm:bW_laplacerep}
The confluent hypergeometric-type function \eqref{eq:bW_def} admits the integral representation
\begin{equation} \label{eq:bW_laplacerep}
\begin{aligned}
\bm{W}_{\!\alpha, \nu}(x) = \int_{-\infty\!}^\infty e^{\nu s} \eta_x(s) ds = 2 \! \int_0^{\infty\!} \cosh(\nu s) \eta_x(s) ds \qquad (\alpha, \nu \in \mathbb{C},\; x > 0)
\end{aligned}
\end{equation}
where $\eta_x$ is the nonnegative function defined by
\[
\eta_x(s) := 2^{-{3 \over 2}} \pi^{-{1 \over 2}} x^{-1 + 2\alpha} \exp\biggl({1 \over 2x^2}-{1 \over 4x^2} \cosh^2\Bigl({s \over 2}\Bigr) \biggr) D_{2\alpha}\biggl( {1 \over x} \cosh\Bigl({s \over 2}\Bigr) \biggr)
\]
being $D_\mu(z)$ the parabolic cylinder function \cite[Chapter VIII]{erdelyiII1953}.
\end{theorem}

\begin{proof}
Only the first equality in \eqref{eq:bW_laplacerep} needs proof. Let us temporarily assume that $\nu \geq 0$ and $-\infty < \alpha < {1 \over 2}$, and let $\xi > 0$. We begin by noting the identity
\begin{equation} \label{eq:tilW_laplacerep_pf1}
\xi^{1-2\alpha} \! \int_0^\infty \exp\biggl(-{\xi^2 \over 2}x - {1 \over 2x} \biggr) x^{-2\alpha} \bm{W}_{\alpha,\nu}\bigl(x^{1 \over 2}\bigr) \, dx = 2K_{2\nu}(\xi) = {1 \over 2} \int_{-\infty}^\infty e^{\nu s} \exp\Bigl(-\xi \cosh\Bigl({s \over 2}\Bigr)\Bigr) \, ds.
\end{equation}
which is a consequence of integrals 2.4.18.12 in \cite{prudnikovI1986} and 2.19.4.7 in \cite{prudnikovIII1989}. Here $K_\nu(x)$ denotes the modified Bessel function of the second kind \cite[Chapter VII]{erdelyiII1953}. To deduce the theorem from this identity, we will use the injectivity property of the Laplace transform, after rewriting the right-hand side as an iterated integral. To that end, we point out that, according to integral 2.11.4.4 in \cite{prudnikovII1986}, for $s, \xi > 0$ we have
\[
\xi^{2\alpha - 1} \exp\Bigl(-\xi \cosh\Bigl({s \over 2}\Bigr)\Bigr) = (2\pi)^{-{1 \over 2}} \! \int_0^\infty \exp\biggl(-{\xi^2 \over 2}x - {1 \over 4x} \cosh^2\Bigl({s \over 2}\Bigr) \biggr) x^{-{1 \over 2} - \alpha} D_{2\alpha}\Bigl( x^{-{1 \over 2}} \cosh\Bigl({s \over 2}\Bigr) \Bigr) dx
\]
Substituting in \eqref{eq:tilW_laplacerep_pf1} and interchanging the order of integration (which is valid since $D_\mu(y) > 0$ for $y > 0$ and $\mu < 1$, cf.\ \cite[Equation 12.5.3]{dlmf}, and therefore the integrand is positive), we find that
\begin{align*}
& \int_0^\infty \exp\biggl(-{\xi^2 \over 2}x - {1 \over 2x} \biggr) x^{-2\alpha\,} \bm{W}_{\alpha,\nu}\bigl(x^{1 \over 2}\bigr) \, dx = \\
& \qquad\quad = 2^{-{3 \over 2}} \pi^{-{1 \over 2}} \! \int_0^\infty \exp\biggl(-{\xi^2 \over 2}x \biggr) x^{-{1 \over 2} - \alpha} \! \int_{-\infty}^\infty \exp\biggl(\nu s - {1 \over 4x} \cosh^2\Bigl({s \over 2}\Bigr)\biggr) D_{2\alpha}\Bigl( x^{-{1 \over 2}} \cosh\Bigl({s \over 2}\Bigr) \Bigr) ds \, dx
\end{align*}
Given that the Laplace transform is one-to-one, this identity yields
\[
e^{-{1 \over 2x}} \bm{W}_{\alpha,\nu}\bigl(x^{1 \over 2}\bigr) = 2^{-{3 \over 2}} \pi^{-{1 \over 2}} x^{-{1 \over 2} + \alpha} \! \int_{-\infty}^\infty \exp\biggl(\nu s - {1 \over 4x} \cosh^2\Bigl({s \over 2}\Bigr)\biggr) D_{2\alpha}\Bigl( x^{-{1 \over 2}} \cosh\Bigl({s \over 2}\Bigr) \Bigr) ds,
\]
finishing the proof for the case $-\infty < \alpha < {1 \over 2}$, $\nu \in \mathbb{R}$.

To extend \eqref{eq:bW_laplacerep} to all $\alpha, \nu \in \mathbb{C}$, it is enough to show that $\int_{-\infty\!}^\infty e^{\nu s} \eta_x(s) ds$ is an entire function of the parameter $\alpha$ and the parameter $\nu$ (so that the usual analytic continuation argument can be applied). For $t > 0$ and $\alpha \in \mathbb{C}$ with $\mathrm{Re}\,\alpha \leq 0$, the integral representation \cite[Equation 12.5.3]{dlmf} gives
\begin{align*}
\bigl|D_{2\alpha}(t)\bigr| & = {e^{-{t^2 \over 4}}\, t^{2\, \mathrm{Re}\,\alpha} \over |\Gamma({1 \over 2} - \alpha)|}  \biggl| \int_0^\infty e^{-s} s^{-{1 \over 2} - \alpha} \Bigl( 1+{2s \over t^2} \Bigr)^{\!\alpha} ds \biggr| \\
& \leq {e^{-{t^2 \over 4}}\, t^{2\, \mathrm{Re}\,\alpha} \over |\Gamma({1 \over 2} - \alpha)|} \int_0^\infty e^{-s} s^{-{1 \over 2} - \mathrm{Re}\,\alpha} \Bigl( 1+{2s \over t^2} \Bigr)^{\!\mathrm{Re}\,\alpha} ds \\
& \leq {e^{-{t^2 \over 4}}\, t^{2\, \mathrm{Re}\,\alpha} \over |\Gamma({1 \over 2} - \alpha)|} \int_0^\infty e^{-s} s^{-{1 \over 2} - \mathrm{Re}\,\alpha} ds \\
& = {\Gamma({1 \over 2} - \mathrm{Re}\,\alpha) \over |\Gamma({1 \over 2} - \alpha)|} e^{-{t^2 \over 4}}\, t^{2\, \mathrm{Re}\,\alpha}.
\end{align*}
Furthermore, from the recurrence relation \cite[Equation 8.2.(14)]{erdelyiII1953} it follows that for each $n \in \mathbb{N}_0$ we have $D_{2\alpha + n}(t) = p_{n,\alpha}(t) D_{2\alpha}(t) + q_{n,\alpha}(t) D_{2\alpha-1}(t)$, being $p_{n,\alpha}(\cdot)$, $q_{n,\alpha}(\cdot)$ polynomials of degree at most $n$ whose coefficients are continuous functions of $\alpha$. It is easy to see that $|p_{n,\alpha}(t)|, |q_{n,\alpha}(t)| \leq C_n(\alpha) \, (1+t^n)$ for some function $C_n(\alpha)$ that depends continuously on $\alpha \in \mathbb{C}$ and, consequently,
\begin{align}
\nonumber & \bigl|D_{2\alpha+n}(t)\bigr| \leq C_n(\alpha) \, (1+t^n) \bigl[D_{2\alpha}(t) + D_{2\alpha-1}(t)\bigr] \\
\nonumber & \qquad\qquad\quad\! \leq C_n(\alpha) \, e^{-{t^2 \over 4}}\, t^{2\, \mathrm{Re}\,\alpha} (1+t^n) \biggl[{\Gamma({1 \over 2} - \mathrm{Re}\,\alpha) \over |\Gamma({1 \over 2} - \alpha)|} + {\Gamma(1 - \mathrm{Re}\,\alpha) \over |\Gamma(1 - \alpha)|} t^{-1} \biggr], \\[2pt]
\label{eq:tilW_laplacerep_pf2} & \sup_{\substack{|\alpha| \leq M \\ \mathrm{Re}\, \alpha \leq 0}} \bigl|D_{2\alpha+n}(t)\bigr| \leq C_{M,n} \, e^{-{t^2 \over 4}} (t^{-2M-1} + t^n)
\end{align}
where $M > 0$ and $n \in \mathbb{N}_0$ are arbitrary and the constant $C_{M,n}$ depends on $M$ and $n$. Using \eqref{eq:tilW_laplacerep_pf2}, we see that
\[
\sup_{(\alpha,\nu) \in \mathcal{R}_{\mathsmaller{M}}} \int_{-\infty}^\infty \biggl| \, \exp\biggl(\nu s -{1 \over 4x} \cosh^2\Bigl({s \over 2}\Bigr) \biggr) D_{2\alpha+n}\Bigl( x^{-{1 \over 2}\!} \cosh\Bigl({s \over 2}\Bigr) \Bigr) \biggr| \, ds < \infty
\]
where $\mathcal{R}_{\mathsmaller{M}} = \bigl\{ (\alpha,\nu): |\alpha| \leq M,\; \mathrm{Re}\,\alpha \leq 0,\; |\nu| \leq M \bigr\}$. Applying the standard results on the analyticity of parameter-dependent integrals (e.g.\ \cite{mattner2001}), we obtain the entireness in $\alpha$ and in $\nu$ of  $\int_{-\infty\!}^\infty e^{\nu s} \eta_x(s) ds$, completing the proof.
\end{proof}

In what follows we again assume that $\alpha < {1 \over 2}$. A straightforward consequence of Theorem \ref{thm:bW_laplacerep} is that $|\bm{W}_{\!\alpha,\nu}(x)| \leq \bm{W}_{\!\alpha,\nu_0}(x)$ whenever $|\mathrm{Re}\, \nu| \leq \nu_0$ ($\nu_0 \geq 0$). Together with \eqref{eq:bW_equal1_param}, this implies that
\begin{equation} \label{eq:bW_ineq_param}
|\bm{W}_{\!\alpha,\nu}(x)| \leq 1 \qquad\;\; \text{for all } x > 0 \text{ and } \nu \text{ in the strip } |\mathrm{Re}\,\nu| \leq \tfrac{1}{2} - \alpha.
\end{equation}
The following lemma will also be useful:

\begin{lemma} \label{lem:bW_ineq_1minusW}
The confluent hypergeometric-type function \eqref{eq:bW_def} satisfies the inequality
\[
1-\bm{W}_{\!\alpha,\nu}(x) \leq 2\bigl((\tfrac{1}{2}-\alpha)^2 - \nu^2\bigr)	\, x^2 \qquad \text{for each } x \geq 0 \text{ and } \nu \in [0,\tfrac{1}{2} - \alpha] \, \cup \, i\, \mathbb{R}.
\]
\end{lemma}

\begin{proof}
Given that $\bm{W}_{\!\alpha,\nu}(\cdot)$ solves the equation $\mathcal{L} w = \bigl((\tfrac{1}{2}-\alpha)^2 - \nu^2\bigr) w$, we have
\[
-{d \over d\xi}\biggl[\xi^{2} \mathrm{m}(\xi) {d \over d\xi} \bm{W}_{\!\alpha,\nu}(\xi)\biggr] = 4 \bigl((\tfrac{1}{2}-\alpha)^2 - \nu^2\bigr) \, \mathrm{m}(\xi) \bm{W}_{\!\alpha,\nu}(\xi).
\]
Taking into account \eqref{eq:bW_limderivatives}, after integrating both sides between $0$ and $y$ and then between $0$ and $x$ we obtain
\begin{equation} \label{eq:bW_integraleq}
1 - \bm{W}_{\!\alpha,\nu}(x) = 4 \bigl((\tfrac{1}{2}-\alpha)^2 - \nu^2\bigr) \! \int_0^x {1 \over y^{2\,} \mathrm{m}(y)} \! \int_0^y \mathrm{m}(\xi) \, \bm{W}_{\!\alpha,\nu}(\xi)\, d\xi\, dy.
\end{equation}
Using \eqref{eq:bW_ineq_param} and the inequality $({\xi \over y})_{\,}^{4(1-\alpha)} \leq 1$ (which holds for $0 < \xi \leq y$ due to the assumption $\alpha < {1 \over 2}$), we thus find that
\begin{align*}
1 - \bm{W}_{\!\alpha,\nu}(x) & \leq 4 \bigl((\tfrac{1}{2} - \alpha)^2 - \nu^2\bigr) \! \int_0^x y^{4\alpha-3} \exp\bigl(\tfrac{1}{2y^2}\bigr) \! \int_0^y \xi^{1-4\alpha} \exp\bigl(-\tfrac{1}{2\xi^2}\bigr) \, d\xi\, dy \\
& \leq 4 \bigl((\tfrac{1}{2} - \alpha)^2 - \nu^2\bigr) \! \int_0^x y \exp\bigl(\tfrac{1}{2y^2}\bigr) \! \int_0^y \xi^{-3} \exp\bigl(-\tfrac{1}{2\xi^2}\bigr) \, d\xi\, dy \\
& = 2 \bigl((\tfrac{1}{2} - \alpha)^2 - \nu^2\bigr)\, x^2
\end{align*}
as required.
\end{proof}

\section{The index Whittaker convolution algebra} \label{chap:iw_convalgebra}

\subsection{Index Whittaker transforms of measures} \label{sec:iw_measures}

The confluent hypergeometric-type function \eqref{eq:bW_def} is the kernel of (the modified form of) the \emph{index Whittaker transform} defined by
\begin{equation} \label{eq:iw_transf}
\widehat{f}(\lambda) \equiv (\bm{\mathrm{W}}f)(\lambda) = \int_0^\infty f(x) \bm{W}_{\!\alpha,\Delta_{\lambda\!}}(x)\, \mathrm{m}(x) dx, \qquad \lambda \geq 0
\end{equation}
where $\Delta_\lambda := \sqrt{(\frac{1}{2}-\alpha)^2-\lambda}\,$ and $\mathrm{m}$ is the function given in \eqref{eq:iw_mweight}. The basic $L_2$-property of the index Whittaker transform is given in the next theorem. In the statement (and later in this work) we use the notation $L_p(\mathrm{m}) := L_p\bigl((0,\infty); \mathrm{m}(x) dx\bigr)$ for the Lebesgue spaces with respect to the weight defined in \eqref{eq:iw_mweight}, whose norm will be denoted by $\|\cdot\|_p$.

\begin{theorem} \label{thm:iw_transf_L2iso}
For $\alpha > 0$, the index Whittaker transform \eqref{eq:iw_transf} defines an isometric isomorphism
\[
\bm{\mathrm{W}}\!: L_2(\mathrm{m}) \longrightarrow L_2\bigl(\Lambda; \rho(\lambda) d\lambda \bigr)
\]
where $\Lambda := \bigl(({1 \over 2} - \alpha)^2,\infty\bigr)$ and $\rho(\lambda) := 2^{1-2\alpha} \pi^{-2} \sinh(-2\pi i \Delta_{\lambda})  \bigl| \Gamma\bigl({1 \over 2} - \alpha + \Delta_\lambda\bigr) \bigr|^{2\!}$, whose inverse is given by
\begin{equation} \label{eq:iw_inverse}
(\bm{\mathrm{W}}^{-1} \phi)(x) = \int_{({1 \over 2} - \alpha)^{2\!}}^\infty \! \phi(\lambda) \, \bm{W}_{\!\alpha, \Delta_{\lambda\!}}(x) \, \rho(\lambda) d\lambda
\end{equation}
the convergence of the integrals \eqref{eq:iw_transf} and \eqref{eq:iw_inverse} being understood with respect to the norm of the spaces $L_2\bigl(\Lambda; \rho(\lambda) d\lambda \bigr)$ and $L_2(\mathrm{m})$ respectively. \end{theorem}

\begin{proof}
The index Whittaker transform in the form \eqref{eq:iw_transf} can be written as $(\bm{\mathrm{W}} f)(\lambda) = 2^{\alpha - 1} [\mathrm{W}_\alpha (\Theta f)](i\Delta_{\lambda})$, where $\Theta: L_2(\mathrm{m}) \longrightarrow L_2\bigl((0,\infty); x^{-2} dx\bigr)$ is the isometric operator defined by
\[
(\Theta f)(x) := 2^{\alpha - 1} x^\alpha e^{-{x \over 2}} f\bigl((2x)^{-1/2}\bigr), \qquad x > 0
\]
and $\mathrm{W}_\alpha$ is the operator of the index Whittaker transform in its classical form, defined by \eqref{eq:intro_Whit_transf}. Therefore, the fact that $\bm{\mathrm{W}}$ is an isomorphism and the inversion formula follows from known results on the $L_2$-theory for the index Whittaker transform, cf.\ \cite[Section 3]{srivastava1998}.
\end{proof}

We will also need the following addenda to Theorem \ref{thm:iw_transf_L2iso}:

\begin{lemma} \label{lem:iw_transf_compactsupp}
Let $f \in \mathrm{C}_\mathrm{c}^2(0,\infty)$, and let $\widehat{f}$ be its index Whittaker transform \eqref{eq:iw_transf}. Then 
\begin{equation} \label{eq:iw_transf_compactsupp}
f(x) = \int_{({1 \over 2} - \alpha)^{2\!}}^\infty \! \widehat{f}(\lambda) \, \bm{W}_{\!\alpha, \Delta_{\lambda\!}}(x) \, \rho(\lambda) d\lambda
\end{equation}
where the right-hand side integral converges absolutely for each $x > 0$.
\end{lemma}

\begin{proof}
For simplicity, write $\lambda = \tau^2 + ({1 \over 2} - \alpha)^2$ with $\tau \geq 0$. Since $f$ has compact support, we can apply the asymptotic expansion \eqref{eq:bW_itau_asymp} and obtain
\begin{align*}
\widehat{f}\bigl(\tau^2 + (\tfrac{1}{2} - \alpha)^2\bigr) & = \tau^{\alpha - {1 \over 2}} \exp\Bigl(-{\pi \tau \over 2}\Bigr) \, O\biggl(\int_0^\infty \! f(y) \, y^{-2\alpha} \exp\bigl( -\tfrac{1}{4y^2} \bigr) \cos\bigl[\tau\log(8\tau y^2) - \tau - \tfrac{\pi}{2}(\tfrac{1}{2} - \alpha) \bigr] dy \biggr) \\
& = \tau^{\alpha - {1 \over 2}} \exp\Bigl(-{\pi \tau \over 2}\Bigr) \, O\biggl(\int_0^\infty \! f(y) \, y^{-2\alpha} \exp\bigl( -\tfrac{1}{4y^2} + 2i\tau\log y \bigr) dy \biggr) \\
& = \tau^{\alpha - {1 \over 2}} \exp\Bigl(-{\pi \tau \over 2}\Bigr) \, O\biggl(\int_{-\infty}^\infty \! g(\xi) \exp\bigl( i\tau\xi \bigr) d\xi \biggr) \\
& = \tau^{\alpha - {5 \over 2}} \exp\Bigl(-{\pi \tau \over 2}\Bigr) \, O\biggl(\int_{-\infty}^\infty \! g''(\xi) \exp\bigl( i\tau\xi \bigr) d\xi \biggr), \hspace{1.5cm} \tau \to \infty
\end{align*}
where $g(\xi) = f(e^{\xi \over 2}) \exp\bigl(\xi({1 \over 2} - \alpha) - {1 \over 4}e^{-\xi}\bigr)$; the last step is obtained using integration by parts, noting that $g \in \mathrm{C}_\mathrm{c}^2(\mathbb{R})$. Consequently, $\widehat{f}\bigl(\tau^2 + (\tfrac{1}{2} - \alpha)^2\bigr) = O\bigl(\tau^{\alpha - {5 \over 2}} \exp\bigl(-\tfrac{\pi \tau}{2}\bigr)\bigr)$ as $\tau \to \infty$. Combining this with \eqref{eq:bW_itau_asymp} and the expansion $\Gamma(a+i\tau) \sim (2\pi)^{1 \over 2} \tau^{a - {1 \over 2}} e^{-{\pi\tau \over 2}}$ as $\tau \to \infty$ \cite[Equation 5.11.9]{dlmf}, we conclude that
\[
\widehat{f}\bigl(\tau^2 + (\tfrac{1}{2} - \alpha)^2\bigr) W_{\alpha, i\tau}(x) \, \tau\rho\bigl(\tau^2 + (\tfrac{1}{2} - \alpha)^2\bigr) = O(\tau^{-2}), \qquad \tau \to \infty,
\]
which proves the absolute convergence of the integral in the right-hand side of \eqref{eq:iw_transf_compactsupp}. The proof is finished by applying Theorem \ref{thm:iw_transf_L2iso}.
\end{proof}

The index Whittaker transform operator extends in a natural way to finite measures:

\begin{definition}
Let $\mu \in \mathcal{M}_{b}[0,\infty)$. The \emph{index Whittaker transform} of the measure $\mu$ is the function defined by the integral
\begin{equation} \label{eq:iw_transf_meas_def}
\widehat{\mu}(\lambda) = \int_{[0,\infty)} \! \bm{W}_{\!\alpha,\Delta_{\lambda\!}}(x)\, \mu(dx), \qquad \lambda \geq 0.
\end{equation}
\end{definition}

The next proposition contains some basic properties of the index Whittaker transform of finite measures. In the statement, the notation $\mu_n \warrow \mu$ refers to the weak convergence of the corresponding measures \cite[Definition 13.12(i)]{klenke2014}.

\begin{proposition} \label{prop:iwmeas_props}
The index Whittaker transform $\widehat{\mu}$ of $\mu \in \mathcal{M}_{b}[0,\infty)$ has the following properties:
\begin{enumerate}[itemsep=1.2pt,topsep=4pt]
\item[\textbf{(i)}] $\widehat{\mu}$ is uniformly continuous on $[0,\infty)$. Moreover, if a family of measures $\{\mu_j: j \in J\} \subseteq \mathcal{P}[0,\infty)$ is such that the family of restricted measures $\bigl\{\mu_j\restrict{(0,\infty)\!\!\!} : j \in J\bigr\}$ is tight, then $\{\widehat{\mu_j}: j \in J\}$ is uniformly equicontinuous on $[0,\infty)$. \\[-10pt]

\item[\textbf{(ii)}] Each measure $\mu \in \mathcal{M}_{b}[0,\infty)$ is uniquely determined by its index Whittaker transform $\widehat{\mu}$.\\[-10pt]

\item[\textbf{(iii)}] \label{prop:iwmeas_props_weak_iwunif} If $\{\mu_n\}$ is a sequence of measures belonging to $\mathcal{M}_{b}[0,\infty)$, $\mu \in \mathcal{M}_{b}[0,\infty)$, and $\mu_n \warrow \mu$, then
\[
\widehat{\mu_n} \xrightarrow[\,n \to \infty\,]{} \widehat{\mu} \qquad \text{uniformly for } \lambda \text{ in compact sets.}
\]

\item[\textbf{(iv)}] \label{prop:iwmeas_props_iwpoint_weak} If $\{\mu_n\}$ is a sequence of measures belonging to $\mathcal{M}_{b}[0,\infty)$ whose index Whittaker transforms are such that
\begin{equation} \label{eq:iwmeas_continuity_hyp}
\widehat{\mu_n}(\lambda) \xrightarrow[\,n \to \infty\,]{} f(\lambda) \qquad \text{pointwise in } \lambda \geq 0
\end{equation}
for some real-valued function $f$ which is continuous at a neighborhood of zero, then $\mu_n \warrow \mu$ for some measure $\mu \in \mathcal{M}_{b}[0,\infty)$ such that $\widehat{\mu} \equiv f$. 
\end{enumerate}
\end{proposition}

\begin{proof}
\textbf{(i)} Let us prove the second statement, which implies the first. Fix $\eps > 0$. By the tightness assumption, we can choose $M > 0$ such that $\mu_j\bigl((0,{1 \over M}) \cup (M,\infty)\bigr) < \eps$. Moreover, noting that $|\mathrm{Re}\,\Delta_\lambda| \leq {1 \over 2} - \alpha$, it is easily seen that $|\exp(\Delta_{\lambda_1} s) - \exp(\Delta_{\lambda_2} s)| \leq |\Delta_{\lambda_1} - \Delta_{\lambda_2}| s \, e^{({1 \over 2} - \alpha) s}$ for all $s, \lambda_1, \lambda_2 \geq 0$ and, consequently, from Theorem \ref{thm:bW_laplacerep} we get
\begin{equation} \label{eq:iwmeas_props_pf1}
\bigl|\bm{W}_{\!\alpha,\Delta_{\lambda_1}\!}(x) - \bm{W}_{\!\alpha,\Delta_{\lambda_2}\!}(x)\bigr| \leq |\Delta_{\lambda_1} - \Delta_{\lambda_2}| \int_{-\infty}^\infty s \, e^{({1 \over 2} - \alpha) s} \eta_x(s) \, ds
\end{equation}
where the integral on the right-hand side converges uniformly with respect to $x$ in compact subsets of $(0,\infty)$ and is therefore a continuous function of $x > 0$. By continuity of $\lambda \mapsto \Delta_\lambda$, we can choose $\delta > 0$ such that 
\begin{equation} \label{eq:iwmeas_props_pf2}
|\Delta_{\lambda_1} - \Delta_{\lambda_2}| < {\eps \over C_M} \;\;\; \text{whenever} \;\; |\lambda_1 - \lambda_2| < \delta \;\;\;\; (\lambda_1,\lambda_2 \geq 0)
\end{equation}
where $C_M = \max_{x \in [{1 \over M}, M]} \int_{-\infty}^\infty s \, e^{({1 \over 2} - \alpha) s} \eta_x(s) \, ds < \infty$. Combining \eqref{eq:iwmeas_props_pf1}--\eqref{eq:iwmeas_props_pf2}, \eqref{eq:bW_ineq_param} and the fact that $\bm{W}_{\!\alpha,\Delta_{\lambda}\!}(0)$ $\equiv 1$, we deduce that
\begin{align*}
& \bigl|\widehat{\mu_j}(\lambda_1) - \widehat{\mu_j}(\lambda_2)\bigr| = \biggl| \int_{(0,\infty)} \bigl(\bm{W}_{\!\alpha,\Delta_{\lambda_1}\!}(x) - \bm{W}_{\!\alpha,\Delta_{\lambda_2}\!}(x)\bigr) \mu_j(dx) \biggr| \\
& \quad \leq\int_{(0,{1 \over M}) \cup (M,\infty)\!} \bigl|\bm{W}_{\!\alpha,\Delta_{\lambda_1}\!}(x) - \bm{W}_{\!\alpha,\Delta_{\lambda_2}\!}(x)\bigr|\mu_j(dx) + \int_{[{1 \over M},M]\!} \bigl|\bm{W}_{\!\alpha,\Delta_{\lambda_1}\!}(x) - \bm{W}_{\!\alpha,\Delta_{\lambda_2}\!}(x)\bigr|\mu_j(dx) \\
& \quad \leq 2\eps + \eps = 3\eps
\end{align*}
for all $j \in J$, provided that $|\lambda_1 - \lambda_2| < \delta$, which means that $\{\widehat{\mu_j}\}$ is uniformly equicontinuous. \\[-10pt]

\textbf{(ii)} Writing $\lambda = \tau^2 + ({1 \over 2} - \alpha)^2$ with $\tau \geq 0$, the index Whittaker transform $\widehat{\mu}(\tau^2 + (\tfrac{1}{2} - \alpha)^2)$ can be written as
\begin{align}
\nonumber \widehat{\mu}\bigl(\tau^2 + (\tfrac{1}{2} - \alpha)^2\bigr) & = {2^{1+2\alpha} \over |\Gamma({1 \over 2} - \alpha + i\tau)|^2} \int_{[0,\infty)} \! \int_0^\infty \exp\bigl(-\tfrac{(xt)^2}{2}\bigr) t^{-2\alpha} K_{2i\tau}(t)\, dt\, \mu(dx) \\
\label{eq:iwmeas_props_pf3} & = {2^{1+2\alpha} \over |\Gamma({1 \over 2} - \alpha + i\tau)|^2} \int_0^\infty \! K_{2i\tau}(t) \, t^{-2\alpha} \! \int_{[0,\infty)} \exp\bigl(-\tfrac{(xt)^2}{2}\bigr) \mu(dx)\, dt
\end{align}
where we have applied integral 2.16.8.3 in \cite{prudnikovII1986}, and the change of order of integration is easily justified. Suppose that $\widehat{\mu_1}(\lambda) = \widehat{\mu_2}(\lambda)$ for all $\lambda \geq ({1 \over 2} - \alpha)^2$. Then \eqref{eq:iwmeas_props_pf3}, together with the injectivity of the Kontorovich-Lebedev transform (see \cite[Theorem 6.5]{yakubovichluchko1994}), imply that 
\[
\int_0^\infty \exp\bigl(-\tfrac{(xt)^2}{2}\bigr) \mu_1(dx) = \int_0^\infty \exp\bigl(-\tfrac{(xt)^2}{2}\bigr) \mu_2(dx) \quad\; \text{for almost every } t>0.
\]
In fact, by continuity this equality holds for all $t \geq 0$, because the integrals converge uniformly with respect to $t \geq 0$. Consequently,
\[
\int_0^\infty e^{-ys} \bm{\mu}_1(dy) = \int_0^\infty e^{-ys} \bm{\mu}_2(dy) \quad\; \text{for all } s \geq 0
\]
where $\bm{\mu}_i$ ($i = 1,2$) are the measures defined by $\bm{\mu}_i(B) = \mu_i(\{x: x^2 \in B\})$. Since the measures $\bm{\mu}_i$ are uniquely determined by their Laplace transforms \cite[Theorem 15.6]{klenke2014}, we have $\bm{\mu}_1 = \bm{\mu}_2$ and, consequently, $\mu_1 = \mu_2$. \\[-10pt]

\textbf{(iii)} Since $\bm{W}_{\!\alpha,\Delta_{\lambda\!}}(x)$ is continuous and bounded, the pointwise convergence $\widehat{\mu_n}(\lambda) \to \widehat{\mu}(\lambda)$ follows trivially from the Portemanteau theorem (see \cite[Theorem 13.16]{klenke2014}). For the restricted measures, we clearly have $\mu_n\restrict{(0,\infty)\!\!} \warrow \mu\restrict{(0,\infty)\!}$, hence (by Prokhorov's theorem \cite[Theorem 13.29]{klenke2014}) $\{\mu_n\restrict{(0,\infty)\!}\}$ is tight and therefore (by part (i)) $\{\widehat{\mu_n}\}$ is uniformly equicontinuous. Invoking Lemma 15.22 in \cite{klenke2014}, we conclude that the convergence $\widehat{\mu_n} \to \widehat{\mu}$ is uniform for $\lambda$ in compact sets. \\[-10pt]

\textbf{(iv)} We only need to show that the sequence $\{\mu_n\}$ is tight. Indeed, if $\{\mu_n\}$ is tight, then Prokhorov's theorem yields that for any subsequence $\{\mu_{n_k}\}$ there exists a further subsequence $\{\mu_{n_{k_j}}\!\}$ and a finite measure $\mu \in \mathcal{M}_{b}[0,\infty)$ such that $\mu_{n_{k_j}}\!\! \warrow \mu$. Then, due to part (iii) and to \eqref{eq:iwmeas_continuity_hyp}, we have $\widehat{\mu}(\lambda) = f(\lambda)$ for all $\lambda \geq 0$, which implies (by part (ii)) that all such subsequences have the same weak limit; consequently, the sequence $\mu_n$ itself converges weakly to $\mu$.

To prove the tightness, take $\eps > 0$. Since $f$ is continuous at a neighborhood of zero, we have ${1 \over \delta} \int_0^{2\delta} \bigl(f(0) - f(\lambda)\bigr)d\lambda \longrightarrow 0$ as $\delta \downarrow 0$; therefore, we can choose $\delta > 0$ such that
\[
{1 \over \delta} \int_0^{2\delta} \bigl(f(0) - f(\lambda)\bigr)d\lambda < \eps.
\]
Next we observe that, as a consequence of \eqref{eq:bW_liminfty} and the dominated convergence theorem, we have $\int_0^{2\delta} \bigl( 1-\bm{W}_{\!\alpha,\Delta_{\lambda\!}}(x) \bigr) d\lambda \longrightarrow 2\delta$ as $x \nearrow \infty$, meaning that we can pick $M>0$ such that
\[
\int_0^{2\delta} \bigl( 1-\bm{W}_{\!\alpha,\Delta_{\lambda\!}}(x) \bigr) d\lambda \geq \delta \qquad \text{for all } x > M.
\]

By our choice of $M$ and Fubini's theorem,
\begin{align*}
\mu_n\bigl([M,\infty)\bigr) & = {1 \over \delta} \int_M^\infty \delta\, \mu_n(dx) \\
& \leq {1 \over \delta} \int_M^\infty \int_0^{2\delta} \bigl( 1-\bm{W}_{\!\alpha,\Delta_{\lambda\!}}(x) \bigr) d\lambda\, \mu_n(dx) \\
& \leq {1 \over \delta} \int_0^\infty \int_0^{2\delta} \bigl( 1-\bm{W}_{\!\alpha,\Delta_{\lambda\!}}(x) \bigr) d\lambda\, \mu_n(dx) \\
& = {1 \over \delta} \int_0^{2\delta} \bigl(\widehat{\mu_n}(0) - \widehat{\mu_n}(\lambda)\bigr) d\lambda.
\end{align*}
Hence, using the dominated convergence theorem,
\begin{align*}
\limsup_{n \to \infty} \mu_n\bigl([M,\infty)\bigr) & \leq {1 \over \delta} \limsup_{n \to \infty}\! \int_0^{2\delta} \bigl(\widehat{\mu_n}(0) - \widehat{\mu_n}(\lambda)\bigr) d\lambda \\
& = {1 \over \delta} \int_0^{2\delta}\!\! \lim_{n \to \infty} \bigl(\widehat{\mu_n}(0) - \widehat{\mu_n}(\lambda)\bigr) d\lambda = {1 \over \delta} \int_0^{2\delta} \bigl(f(0) - f(\lambda)\bigr) d\lambda < \eps
\end{align*}
due to the choice of $\delta$. Since $\eps$ is arbitrary, we conclude that $\{\mu_n\}$ is tight, as desired.
\end{proof}

\begin{remark}
Parts (iii) and (iv) of the proposition above show that the index Whittaker transform possesses the following important property: \emph{the index Whittaker transform is a topological homeomorphism between $\mathcal{P}[0,\infty)$ with the weak topology and the set $\widehat{\mathcal{P}}$ of index Whittaker transforms of probability measures with the topology of uniform convergence in compact sets.}
\end{remark}

\subsection{Index Whittaker translation and convolution} \label{sec:iw_transl_conv}

The next theorem contains the product formula which is the starting point for the construction of the convolution operator associated with the index Whittaker transform.

\begin{theorem} \label{thm:bW_prodform}
The product $\bm{W}_{\alpha,\nu}(x) \bm{W}_{\alpha,\nu}(y)$ of two functions \eqref{eq:bW_def} with different arguments admits the integral representation
\begin{equation} \label{eq:bW_prodform}
\bm{W}_{\alpha,\nu}(x) \bm{W}_{\alpha,\nu}(y) = \int_0^\infty \bm{W}_{\alpha,\nu}(\xi)\, q(x,y,\xi)\, \mathrm{m}(\xi) d\xi \qquad (x,y > 0, \; \alpha, \nu \in \mathbb{C})
\end{equation}
where $\mathrm{m}(\cdot)$ is defined in \eqref{eq:iw_mweight} and
\begin{align*}
q(x,y,\xi) & \equiv q_\alpha(x,y,\xi) \\
& := (2\pi)^{-{1 \over 2}} (xy\xi)^{-1+2\alpha} \exp\biggl( {1 \over 2x^2} + {1 \over 2y^2} + {1 \over 2\xi^2} - \biggl({x^2+y^2+\xi^2 \over 4xy\xi}\biggr)^{\!\!2\,} \biggr) D_{2\alpha}\biggl( {x^2+y^2+\xi^2 \over 2xy\xi} \biggr) \\
& \, = \int_{({1 \over 2} - \alpha)^{2\!}}^\infty \!  \bm{W}_{\!\alpha, \Delta_{\lambda\!}}(x)  \bm{W}_{\!\alpha, \Delta_{\lambda\!}}(y) \bm{W}_{\!\alpha, \Delta_{\lambda\!}}(\xi)\, \rho(\lambda) d\lambda.
\end{align*}
In particular, $\int_0^\infty q(x,y,\xi)\, \mathrm{m}(\xi) d\xi = 1$ for all $x,y > 0$.
\end{theorem}

\begin{proof}
This result follows from \cite[Theorem 3.1]{sousaetal2018} by applying an elementary change of variables. The representation as an index integral is a consequence of \cite[Equation (45)]{sousaetal2018}. The last statement is obtained by setting $\nu = {1 \over 2} - \alpha$ and recalling \eqref{eq:bW_equal1_param}.
\end{proof}

An upper bound for the kernel of the product formula \eqref{eq:bW_prodform} which will later be useful is
\begin{equation} \label{eq:bW_prodform_kernbound}
|q(x,y,\xi)| \leq {C \over xy\xi} (x^2+y^2+\xi^2)^{2\alpha} \exp\biggl( {1 \over 4x^2} + {1 \over 4\xi^2} - {y^2 \over 8(x\xi)^2} - {(x^2-\xi^2)^2 \over 8(xy\xi)^2}\biggr), \qquad y \in [0,M],\; x,\xi > 0
\end{equation}
where $C > 0$ is a constant depending only on $M$. This bound, which is valid for all $\alpha \in \mathbb{R}$, follows from the inequality ${x^2+y^2+\xi^2 \over 2xy\xi} \geq M^{-1}$ (valid for $x,\xi > 0$,\, $y \in [0,M]$) and the fact that the function $t^{-2\alpha} e^{t^2/4} D_{2\alpha}(t)$ is bounded on $[M^{-1},\infty)$, see \cite[Equation 8.4(1)]{erdelyiII1953}. It is worth noting that $q(x,y,\xi) \equiv q(y,x,\xi) \equiv q(\xi,x,y)$.

In addition, if $\alpha < {1 \over 2}$, we have the positivity condition
\begin{equation} \label{eq:iw_gentransl_positivity}
q(x,y,\xi) > 0 \qquad\quad (x, y, \xi > 0)
\end{equation}
which follows from the properties of the parabolic cylinder function (but would be highly nontrivial to prove if an explicit form for the kernel of the product formula were not available, cf.\ \cite{carlen2010,chebli1995}).

We now define the generalized translation operator induced by \eqref{eq:bW_prodform}, for $\alpha < {1 \over 2}$.

\begin{definition} \label{def:iw_gentransl}
Let $1 \leq p \leq \infty$. The linear operator
\begin{equation} \label{eq:iw_gentransl_def}
(\mathcal{T}^y f)(x) = \int_0^\infty \! f(\xi) q(x,y,\xi)\, \mathrm{m}(\xi) d\xi \qquad \bigl( f \in L_p(\mathrm{m}),\; x, y > 0 \bigr)
\end{equation}
will be called the \emph{index Whittaker translation}.
\end{definition}

As a first remark, we note that the last statement of Theorem \ref{thm:bW_prodform} means that
\begin{equation} \label{eq:iw_gentransl_1preserv}
\mathcal{T}^y \mathds{1} = \mathds{1} \qquad (y > 0)
\end{equation}
where $\mathds{1}$ denotes the function identically equal to one. The properties \eqref{eq:iw_gentransl_positivity} and \eqref{eq:iw_gentransl_1preserv} mean that the index Whittaker translation (and convolution) satisfies the \emph{hypergroup property} as defined e.g.\ in \cite[Section 3.2]{bakryhuet2008}; we however stress that, as discussed in the Introduction, the convolution algebra studied here does not satisfy the axiom scheme of hypergroups (cf.\ \cite{bloomheyer1995} and references therein) under which a comprehensive theory of harmonic analysis has been developed. 

Some important facts on the translation operator \eqref{eq:iw_gentransl_def} are summarized in the following proposition:

\begin{proposition} \label{prop:iw_gentransl_props}
Fix $y > 0$. Then: \\[-8pt]

\textbf{(a)} If $f \in L_\infty(\mathrm{m})$ is such that $0 \leq f \leq 1$, then $0 \leq \mathcal{T}^y f \leq 1$; \\[-8pt]

\textbf{(b)} For each $1 \leq p \leq \infty$, we have
\[
\|\mathcal{T}^y f\|_{p} \leq \|f\|_{p}  \qquad \text{ for all } f \in L_p(\mathrm{m})
\]
(in particular, $\mathcal{T}^y \bigl(L_p(\mathrm{m})\bigr) \subset L_p(\mathrm{m})$); \\[-8pt]

\textbf{(c)} If $f \in L_p(\mathrm{m})$ where $1 < p \leq \infty$, then $\mathcal{T}^y f \in \mathrm{C}(0,\infty)$ and, moreover, we have 
\[
\lim_{h \to 0} \|\mathcal{T}^{y+h} f - \mathcal{T}^y f \|_{p} = 0;
\]

\textbf{(d)} If $f \in \mathrm{C}_\mathrm{b}(0,\infty)$, then $(\mathcal{T}^y f)(x) \to f(y)$ as $x \to 0$; \\[-8pt]

\textbf{(e)} If $f \in L_\infty(\mathrm{m})$ is such that $\lim_{x \to \infty} f(x) = 0$, then $\lim_{x \to \infty} (\mathcal{T}^y f)(x) = 0$.
\end{proposition}

\begin{proof}
All the properties are a direct consequence of the corresponding statements in \cite[Proposition 4.3]{sousaetal2018}, taking into account the elementary connection between the operator $\mathcal{T}^y$ from Definition \ref{def:iw_gentransl} and the translation operator defined in \cite[Definition 4.1]{sousaetal2018}. Alternatively, a direct proof can be given by using similar arguments.
\end{proof}

We observe that, as a consequence of Proposition \ref{prop:iw_gentransl_props}, the index Whittaker translation \eqref{eq:iw_gentransl_def} (with the convention that $(\mathcal{T}^x f)(0) = (\mathcal{T}^0 f)(x) = f(x)$ for all $x$) satisfies the properties
\begin{equation} \label{eq:iw_gentransl_contpreserv}
\mathcal{T}^y \bigl(\mathrm{C}_\mathrm{b}[0,\infty)\bigr) \subset \mathrm{C}_\mathrm{b}[0,\infty) \qquad \text{and} \qquad \mathcal{T}^y \bigl(\mathrm{C}_0[0,\infty)\bigr) \subset \mathrm{C}_0[0,\infty) \qquad (y \geq 0),
\end{equation}
as well as the obvious symmetry property
\begin{equation} \label{eq:iw_gentransl_symm}
(\mathcal{T}^y f)(x) = (\mathcal{T}^x f)(y) \qquad (x,y \geq 0).
\end{equation}
It is also easy to check that the index Whittaker translation is symmetric with respect to the measure $\mathrm{m}(x)dx$, in the sense that for $f,g \in \mathrm{C}_\mathrm{b}[0,\infty) \cap  L_1\bigl(\mathrm{m})$ we have
\begin{equation} \label{eq:iw_gentransl_symmweig}
\int_0^\infty (\mathcal{T}^y f)(x) g(x) \mathrm{m}(x) dx = \int_0^\infty \! f(x) (\mathcal{T}^y g)(x) \mathrm{m}(x) dx.
\end{equation}

We may now define, in the natural way, the generalized convolution associated with the translation operator \eqref{eq:iw_gentransl_def}:

\begin{definition}
Let $\mu, \nu \in \mathcal{M}_{b}[0,\infty)$. The measure $\mu * \nu$ defined by
\begin{equation} \label{eq:iwconv_def}
\int_{[0,\infty)\!} f(x)\, (\mu * \nu)(dx) = \int_{[0,\infty)\!} \int_{[0,\infty)\!} (\mathcal{T}^y f)(x)\, \mu(dx) \nu(dy), \qquad f \in \mathrm{C}_\mathrm{b}[0,\infty)
\end{equation}
is called the \emph{index Whittaker convolution} of the measures $\mu$ and $\nu$.
\end{definition}

Due to \eqref{eq:iw_gentransl_1preserv} and \eqref{eq:iw_gentransl_positivity}, the index Whittaker convolution of two probability measures $\mu, \nu \in \mathcal{P}[0,\infty)$ is also a probability measure. Furthermore, Lemma \ref{lem:iwconv_hatprod} below shows that the index Whittaker convolution is commutative and associative. Consequently: 

\begin{proposition}
The space $(\mathcal{M}_\mathbb{C}[0,\infty),*)$ is an algebra over $\mathbb{C}$ whose identity element is the Dirac measure $\delta_0$.
\end{proposition}

Since $\int_{[0,\infty)\!} f(\xi) (\delta_x * \delta_y)(d\xi) = (\mathcal{T}^y f)(x)$, the fact that $q(x,y,\xi)$ is strictly positive for $x,y,\xi > 0$ yields that $\mathrm{supp}(\delta_x * \delta_y) = [0,\infty)$ for all $x,y > 0$, in sharp contrast with the compactness axiom (H\textsubscript{C}) which is part of the definition of a hypergroup (cf.\ Introduction). It is worth mentioning that positive product formulas which lead to convolution operators not satisfying the hypergroup requirements on $\mathrm{supp}(\delta_x * \delta_y)$ have also been found for certain families of orthogonal polynomials  \cite{connettschwartz1995}.

We now state the fundamental connection between the index Whittaker transform and convolution:

\begin{proposition} \label{lem:iwconv_hatprod}
Let $\mu, \mu_1, \mu_2 \in \mathcal{M}_{b}[0,\infty)$. We have $\mu = \mu_1 * \mu_2$ if and only if
\[
\widehat{\mu}(\lambda) = \widehat{\mu_1}(\lambda)\, \widehat{\mu_2}(\lambda) \qquad \text{for all } \lambda \geq 0. 
\]
\end{proposition}

\begin{proof}
In view of \eqref{eq:bW_prodform}, we have $\bigl(\mathcal{T}^y \bm{W}_{\!\alpha,\Delta_{\lambda\!}}\bigr)(x) = \bm{W}_{\!\alpha,\Delta_{\lambda\!}}(x) \bm{W}_{\!\alpha,\Delta_{\lambda\!}}(y)$, hence
\begin{align*}
\widehat{\mu_1 * \mu_2}(\lambda) & = \int_{[0,\infty)\!\!} \bm{W}_{\!\alpha,\Delta_{\lambda\!}}(x) \, (\mu_1 * \mu_2)(dx) \\
& = \int_{[0,\infty)\!} \int_{[0,\infty)\!} \bigl(\mathcal{T}^y \bm{W}_{\!\alpha,\Delta_{\lambda\!}}\bigr)(x)\, \mu_1(dx) \mu_2(dy)\\
& = \int_{[0,\infty)\!} \int_{[0,\infty)\!\!} \bm{W}_{\!\alpha,\Delta_{\lambda\!}}(x) \bm{W}_{\!\alpha,\Delta_{\lambda\!}}(y) \, \mu_1(dx) \mu_2(dy) \: = \: \widehat{\mu_1}(\lambda) \widehat{\mu_2}(\lambda), \qquad\;\; \lambda \geq 0.
\end{align*}
This proves the ``only if" part, and the converse follows from the uniqueness property in Proposition \ref{prop:iwmeas_props}(ii).
\end{proof}

\subsection{Further properties of the index Whittaker translation}

It is straightforward to show, by a computation similar to that in the proof of Proposition \ref{lem:iwconv_hatprod}, that the index Whittaker translation operator is connected with the index Whittaker transform via the identity 
\begin{equation} \label{eq:iw_gentransl_conn}
\widehat{\mathcal{T}^y f}(\lambda) = \bm{W}_{\!\alpha,\Delta_{\lambda\!}}(y) \widehat{f}(\lambda) \qquad\quad (f \in \mathrm{C}_\mathrm{c}[0,\infty), \; y, \lambda \geq 0).
\end{equation}
From this identity, we obtain the following proposition.

\begin{proposition}
For each $f \in \mathrm{C}_\mathrm{b}[0,\infty)$ and $x,\xi \geq 0$:
\begin{equation} \label{eq:iw_gentransl_symmxy}
\mathcal{T}^\xi \mathcal{T}^y f = \mathcal{T}^y \mathcal{T}^\xi f, \qquad f \in \mathrm{C}_\mathrm{b}[0,\infty).
\end{equation}
\end{proposition}

Next we investigate the properties of the generalized translation of $f \in \mathcal{E}$, where
\begin{equation} \label{eq:iw_transl_Espace_def}
\mathcal{E} := \biggl\{f \in L_0(0,\infty): \; |f(x)| \leq b_{1\,} \exp\biggl(\frac{1}{2x^2} + b_2(x^{-\beta\!} + x^{\beta})\biggr) \; \text{ for some } b_1, b_2 \geq 0 \text{ and } 0 \leq \beta < 2\biggr\}
\end{equation}
being $L_0(0,\infty)$ the space of Lebesgue measurable functions $f:(0,\infty) \to \mathbb{R}$. The reader should note that the condition $f \in \mathcal{E}$ ensures that for each $x, y > 0$ the index Whittaker translation $(\mathcal{T}^y f)(x) = \int_0^\infty \! f(\xi) q(x,y,\xi)\, \mathrm{m}(\xi) d\xi$ exists as an absolutely convergent integral, as can be verified using \eqref{eq:bW_prodform_kernbound}.

\begin{lemma} \label{lem:iw_transl_Espace_lem1}
Fix $y, M > 0$. Let $f \in \mathcal{E}$ and $p \in \mathbb{N}_0$. Then, for each $\eps > 0$ there exists $\delta, M_0 > 0$ such that
\[
\int_{E_{\mathsmaller{M}}}\ \! |f(\xi)|\, \Bigl|{\partial^p \over \partial x^p} q(x,y,\xi)\Bigr| \mathrm{m}(\xi) \, d\xi < \eps \qquad \text{ for all } x \in (0,\delta] \text{ and } M \geq M_0
\]
where $E_{M} = (0,\tfrac{1}{M}] \cup [M,\infty)$.
\end{lemma}

\begin{proof}
Fix $k \geq -1 + \max\{2\alpha,0\}$. Note that (after a new choice of $b_2$ and $\beta$) the function $\xi \mapsto \xi^{k}f(\xi)$ also belongs to $\mathcal{E}$. Let $\delta < {y^2 \over 2}$. If $|\xi^2 - y^2| \geq \delta$ and $x^2 \leq {\delta \over 2}$, using \eqref{eq:bW_prodform_kernbound}, the boundedness of the function $|t|^{k+1} e^{-|t|^2}$ and the inequalities
\begin{align*}
& {(x^2 + y^2 + \xi^2)^{2\alpha} \over |x^2 + \xi^2 - y^2|^{k+1}} \leq \biggl(1+{4y^2 \over \delta}\biggr)^{\!2\alpha}\Bigl({\delta \over 2}\Bigr)^{\!2\alpha - k - 1}, & \hspace{-2cm} \alpha \geq 0 \\[2pt]
& {(x^2 + y^2 + \xi^2)^{2\alpha} \over |x^2 + \xi^2 - y^2|^{k+1}} \leq y^{4\alpha} \Bigl({\delta \over 2}\Bigr)^{\!- k - 1}, & \hspace{-2cm} \alpha \leq 0
\end{align*}
we find that
\begin{equation} \label{eq:iw_transl_Espace_lem1_pf}
\begin{aligned}
x^{-k} |f(\xi)|\, q_\alpha(x,y,\xi) \, \mathrm{m}(\xi) & \leq {C \over (xy\xi)^{k+1}} (x^2+y^2+\xi^2)^{2\alpha} \exp\biggl( b_2(\xi^{-\beta\!} + \xi^{\beta}) - {(x^2+\xi^2-y^2)^2 \over 8(xy\xi)^2}\biggr)  \\[2pt]
& \leq C \, \exp\biggl( b_2(\xi^{-\beta\!} + \xi^{\beta}) - {(x^2+\xi^2-y^2)^2 \over (4xy\xi)^2}\biggr), & \hspace{-2.2cm} |\xi^2 - y^2| \geq \delta, \; x^2 \leq \tfrac{\delta}{2}
\end{aligned}
\end{equation}
where $C$ depends only on $y$ and $\delta$. Since $0 \leq \beta < 2$, the integral $\int_0^\infty \exp\Bigl\{ b_2(\xi^{-\beta\!} + \xi^{\beta}) - {(x^2+\xi^2-y^2)^2 \over (4xy\xi)^2}\Bigr\} d\xi$ converges uniformly in $x \in [0,({\delta \over 2})^{1/2}]$. Combining this with the inequality \eqref{eq:iw_transl_Espace_lem1_pf}, we conclude that $M_0 > 0$ can be chosen so large that
\begin{equation} \label{eq:iw_transl_Espace_lem1_sc}
\int_{E_{\mathsmaller{M}}} \! x^{-k} |f(\xi)|\, q(x,y,\xi) \, \mathrm{m}(\xi) \, d\xi < \delta \qquad \text{ for all } 0 < x < \bigl(\tfrac{\delta}{2}\bigr)^{1/2} \text{ and } M \geq M_0.
\end{equation}
Using the identity
\begin{equation} \label{eq:bW_prodform_kernderiv}
{\partial q_\alpha(x,y,\xi) \over \partial x} = {y^2 + \xi^2 - x^2 \over 2x^3 (y \xi)^2} \, q_{\alpha + {1 \over 2}}(x,y,\xi) - \bigl( x^{-3} + (1-2\alpha) x^{-1} \bigr) \, q_\alpha(x,y,\xi)
\end{equation}
we inductively see that the function $f(\xi)\, {\partial^p \over \partial x^p} q_\alpha(x,y,\xi)$ can be written as a finite sum of the form\linebreak $\sum_j C_j x^{-k_j} g_j(\xi)\, q_{\alpha_j}(x,y,\xi)$, where $g_j \in \mathcal{E}$ and $k_j \geq -1 + \max\{2\alpha_j,0\}$ for all $j$. Therefore, the conclusion of the lemma follows from \eqref{eq:iw_transl_Espace_lem1_sc}.
\end{proof}

\begin{lemma} \label{lem:iw_transl_Espace_limits}
Let $f \in \mathcal{E} \cap \mathrm{C}^2(0,\infty)$ and $y > 0$. Then: 
\begin{enumerate}[itemsep=0pt,topsep=4pt]
\item[\textbf{(i)}] $\lim\limits_{x \to 0} (\mathcal{T}^y f)(x) = f(y)$; 
\item[\textbf{(ii)}] $\lim\limits_{x \to 0} {\partial \over \partial x}(\mathcal{T}^y f)(x) = 0$.
\end{enumerate}
\end{lemma}

\begin{proof}
\textbf{\emph{(i)}} We will first show that it is enough to prove the result for $f \in \mathrm{C}_\mathrm{c}^2(0,\infty)$. Suppose that part \emph{(i)} of the lemma holds for $f \in \mathrm{C}_\mathrm{c}^2(0,\infty)$. Let $g \in \mathcal{E} \cap \mathrm{C}^2(0,\infty)$ and $\eps,M > 0$; then, choose $\delta > 0$ and $g_\mathrm{c} \in \mathrm{C}_\mathrm{c}^2(0,\infty)$ such that $g(\xi) = g_\mathrm{c}(\xi)$ for all $\xi \in [{1 \over M}, M]$ and
\[
|(\mathcal{T}^y g)(x) - (\mathcal{T}^y g_\mathrm{c})(x)| < \eps \qquad \text{for all } x \in (0,\delta]
\]
(to see that this is possible, apply the case $p=0$ of Lemma \ref{lem:iw_transl_Espace_lem1}). If $y \in [{1 \over M},M]$, we obtain
\[
\limsup_{x \to 0} |(\mathcal{T}^y g)(x) - g(y)| \leq \eps + \lim_{x \to 0} |(\mathcal{T}^y g_\mathrm{c})(x) - g_\mathrm{c}(y)| = \eps.
\]
As $M$ and $\eps$ are arbitrary, we conclude that $\lim_{x \to 0} (\mathcal{T}^y g)(x) = g(y)$ for all $y > 0$ and $g \in \mathcal{E} \cap \mathrm{C}^2(0,\infty)$.

Let us now prove that $\lim_{x \to 0} (\mathcal{T}^y f)(x) = f(y)$ holds for $f \in \mathrm{C}_\mathrm{c}^2(0,\infty)$. Using the integral representation for $q(x,y,\xi)$ given in Theorem \ref{thm:bW_prodform}, we write
\begin{equation} \label{eq:iw_transl_indexrep}
\begin{aligned}
(\mathcal{T}^y f)(x) & = \int_0^\infty f(\xi) \int_{({1 \over 2} - \alpha)^{2\!}}^\infty \!  \bm{W}_{\!\alpha, \Delta_{\lambda\!}}(x)  \bm{W}_{\!\alpha, \Delta_{\lambda\!}}(y) \bm{W}_{\!\alpha, \Delta_{\lambda\!}}(\xi)\, \rho(\lambda) d\lambda\, \mathrm{m}(\xi) d\xi \\
& = \int_{({1 \over 2} - \alpha)^{2\!}}^\infty \!  \widehat{f}(\lambda) \, \bm{W}_{\!\alpha, \Delta_{\lambda\!}}(x)  \bm{W}_{\!\alpha, \Delta_{\lambda\!}}(y) \, \rho(\lambda) d\lambda
\end{aligned}
\end{equation}
where the second equality is obtained by changing the order of integration, which is valid because $f$ has compact support. Therefore, \eqref{eq:bW_limderivatives}, \eqref{eq:bW_ineq_param}, Lemma \ref{lem:iw_transf_compactsupp} and the dominated convergence theorem yield that
\[
\lim_{x \to 0} (\mathcal{T}^y f)(x) = \int_{({1 \over 2} - \alpha)^{2\!}}^\infty \!  \widehat{f}(\lambda) \Bigl(\mskip 0.6\thinmuskip\lim_{x \to 0} \bm{W}_{\!\alpha, \Delta_{\lambda\!}}(x)\Bigr) \bm{W}_{\!\alpha, \Delta_{\lambda\!}}(y) \, \rho(\lambda) d\lambda = f(y),
\]
concluding the proof. \\[-7pt]

\textbf{\emph{(ii)}} Identical reasoning as in part (i) shows that it is enough to prove the result for $f \in \mathrm{C}_\mathrm{c}^2(0,\infty)$. Taking $f \in \mathrm{C}_\mathrm{c}^2(0,\infty)$, differentiation of \eqref{eq:iw_transl_indexrep} under the integral sign gives
\[
{\partial \over \partial x} (\mathcal{T}^y f)(x) = \int_{({1 \over 2} - \alpha)^{2\!}}^\infty \!  \widehat{f}(\lambda) \, {d \over dx} \bm{W}_{\!\alpha, \Delta_{\lambda\!}}(x) \, \bm{W}_{\!\alpha, \Delta_{\lambda\!}}(y) \, \rho(\lambda) d\lambda
\]
If we now apply \eqref{eq:bW_limderivatives}, by dominated convergence we conclude that $\lim_{x \to 0} {\partial \over \partial x} (\mathcal{T}^y f)(x) = 0$.
\end{proof}

\section{$*$-infinitely divisible distributions} \label{chap:iw_infdiv}

The set of \emph{$*$-infinitely divisible distributions} is defined as
\[
\mathcal{P}_\mathrm{id} = \bigl\{ \mu \in \mathcal{P}[0,\infty) \bigm| \text{for all } n \in \mathbb{N} \text{ there exists } \nu_n \in \mathcal{P}[0,\infty) \text{ such that } \mu = \nu_n^{*n} \bigr\}
\]
where $\nu_n^{*n}$ denotes the $n$-fold index Whittaker convolution of $\nu_n$ with itself.

\begin{lemma} \label{lem:iwinfdivis_zeros_idempotent}
Let $\mu \in \mathcal{P}_\mathrm{id}$. Then $0 < \widehat{\mu}(\lambda) \leq 1$ for all $\lambda \geq 0$. Moreover, $\mu$ has no nontrivial idempotent divisors, i.e., if $\mu = \vartheta * \nu$ (with $\vartheta, \nu \in \mathcal{P}[0,\infty)$) where $\vartheta$ is idempotent with respect to the index Whittaker convolution (that is, it satisfies $\vartheta = \vartheta * \vartheta$), then $\vartheta = \delta_0$.
\end{lemma}

\begin{proof}
The inequality $\widehat{\mu}(\lambda) \leq 1$ is obvious from \eqref{eq:bW_ineq_param}. The proof of the positivity of $\widehat{\mu}$ relies on the properties of the index Whittaker transform of measures (namely Proposition \ref{prop:iwmeas_props}(iv)) and is identical to the proof of \cite[Lemma 7.5]{sato1999}.

Assume that $\mu = \vartheta * \nu$ with $\vartheta$ idempotent. Then $\bigl(\widehat{\vartheta}(\lambda)\bigr)^2 = \widehat{\vartheta}(\lambda)$ for all $\lambda$, and consequently $\widehat{\vartheta}(\lambda)$ only takes the values $0$ and $1$. However, $\widehat{\mu}(\lambda) = \widehat{\vartheta}(\lambda)\,\widehat{\nu}(\lambda) \neq 0$; hence $\widehat{\vartheta}(\lambda) = 1$ for all $\lambda$, i.e., $\vartheta = \delta_0$.
\end{proof}

The first part of the lemma shows that the index Whittaker transform of $\mu \in \mathcal{P}_\mathrm{id}$ is of the form
\[
\widehat{\mu}(\lambda) = e^{-\psi_\mu(\lambda)}
\]
where $\psi_\mu(\lambda)$ ($\lambda \geq 0$) is a positive continuous function such that $\psi_\mu(0) = 0$, which we shall call the \emph{$*$-exponent} of $\mu$. The next result shows that the $*$-exponent of an infinitely divisible distribution grows at most linearly:

\begin{proposition}
Let $\mu \in \mathcal{P}_\mathrm{id}$. Then
\[
\psi_\mu(\lambda) \leq C_\mu (1+\lambda) \qquad \text{for all } \lambda \geq 0
\]
for some constant $C_\mu > 0$ which is independent of $\lambda$.
\end{proposition}

\begin{proof}
For $n \in \mathbb{N}$, denote by $\nu_n$ the probability measure such that $\mu = \nu_n^{*n}$ (and thus \scalebox{0.99}{$\widehat{\nu_n}(\lambda) \equiv \exp(-{1 \over n} \psi_\mu(\lambda))$}). Due to the inequality $1-e^{-\tau} \leq \tau$ (valid for $\tau \geq 0$) and the fact that $\lim_n n(1-e^{-k/n}) = k$ for each $k \in \mathbb{R}$, we have
\begin{equation} \label{eq:iwmeas_infdiv_ineqlim}
n\bigl(1-\widehat{\nu_n}(\lambda)\bigr) \leq \psi_\mu(\lambda) \text{ for all } n \in \mathbb{N}, \qquad\quad \lim_{n \to \infty} n\bigl(1-\widehat{\nu_n}(\lambda)\bigr) = \psi_\mu(\lambda).
\end{equation}

Pick $\lambda_1 > 0$. By \eqref{eq:bW_limderivatives}, there exists $\eps > 0$ such that ${d^2 \over dx^2} \bm{W}_{\!\alpha,\Delta_{\lambda_1\!\!}}(x) \leq -2\lambda_1$ for all $0 < x \leq \eps$, and then we have
\begin{equation} \label{eq:iwmeas_infdiv_bound_pf0}
{1 \over \lambda_1} \bigl(1 - \bm{W}_{\!\alpha,\Delta_{\lambda_1\!\!}}(x)\bigr) = {1 \over \lambda_1} \int_0^x  (y-x)\, {d^2 \over dx^2} \bm{W}_{\!\alpha,\Delta_{\lambda_1\!\!}}(y) \, dy \geq x^2 \qquad \text{for all } 0 \leq x \leq \eps.
\end{equation}
Using also Lemma \ref{lem:bW_ineq_1minusW}, we get
\begin{equation} \label{eq:iwmeas_infdiv_bound_pf1}
\begin{aligned}
n \int_{[0,\eps)\!} \bigl(1-\bm{W}_{\!\alpha,\Delta_\lambda\!}(x)\bigr) \nu_n(dx) & \leq 2\lambda n \int_{[0,\eps)\!} x^2\, \nu_n(dx) \\
& \leq {2\lambda n \over \lambda_1} \int_{[0,\eps)} \bigl(1 - \bm{W}_{\!\alpha,\Delta_{\lambda_1\!\!}}(x)\bigr) \nu_n(dx) \\
& \leq {2\lambda n \over \lambda_1} \bigl(1-\widehat{\nu_n}(\lambda_1)\bigr) \leq {2\lambda \over \lambda_1} \psi_\mu(\lambda_1).
\end{aligned}
\end{equation}

Next, from the asymptotic expansion given in \eqref{eq:bW_itau_asymp} we easily see that there exists $\lambda_2 > 0$ such that
\[
\bigl| \bm{W}_{\!\alpha,\Delta_{\lambda_{2\!\!}}}(x) \bigr| \leq {1 \over 2} \qquad \text{for all } x \geq \eps
\]
and using \eqref{eq:bW_ineq_param} we obtain
\begin{equation} \label{eq:iwmeas_infdiv_bound_pf2}
\begin{aligned}
n \int_{[\eps,\infty)\!} \bigl(1-\bm{W}_{\!\alpha,\Delta_\lambda\!}(x)\bigr) \nu_n(dx) & \leq 2n \int_{[\eps,\infty)\!} \nu_n(dx) \\
& \leq 4n \int_{[\eps,\infty)\!}\bigl( 1-\bm{W}_{\!\alpha,\Delta_{\lambda_{2\!\!}}}(x) \bigr) \nu_n(dx) \\
& \leq 4n\bigl(1-\widehat{\nu_n}(\lambda_2)\bigr) \leq 4\psi_\mu(\lambda_2).
\end{aligned}
\end{equation}
Combining \eqref{eq:iwmeas_infdiv_bound_pf1} and \eqref{eq:iwmeas_infdiv_bound_pf2} one sees that
\begin{equation} \label{eq:iwmeas_infdiv_bound_pf3}
n(1-\widehat{\nu_n}(\lambda)) = n \int_{[0,\infty)\!} \bigl(1-\bm{W}_{\!\alpha,\Delta_{\lambda_{2\!\!}}}(x)\bigr) \nu_n(dx) \leq {2\lambda \over \lambda_1} \psi_\mu(\lambda_1) + 4\psi_\mu(\lambda_2) \leq C_\mu (1+\lambda)
\end{equation}
where $C_\mu = \max\bigl\{{2 \over \lambda_1} \psi_\mu(\lambda_1), 4\psi_\mu(\lambda_2) \bigr\}$. The inequality \eqref{eq:iwmeas_infdiv_bound_pf3} holds for all $n$; taking the limit $n \to \infty$ yields
\[
\psi_\mu(\lambda) \leq C_\mu (1+\lambda)
\]
which completes the proof.
\end{proof}

The compound Poisson and the Gaussian distributions are the two probability measures which are in the center of the celebrated Lévy-Khintchine formula on the characterization of the infinitely divisible distributions with respect to the classical convolution. We will now define their counterparts with respect to the index Whittaker convolution.

\begin{definition}
Let $\mu \in \mathcal{P}[0,\infty)$ and $a>0$. The measure $\mb{e}(a\mu)$ defined by
\[
\mb{e}(a\mu) = e^{-a} \sum_{k=0}^\infty {a^k \over k!} \mu^{*k}
\]
(the infinite sum converging in the weak topology) is said to be the \emph{$*$-compound Poisson measure} associated with $a\mu$.
\end{definition}

This definition is completely analogous to that of the classical compound Poisson measure. From the definition it immediately follows that $\mb{e}(a\mu) \in \mathcal{P}[0,\infty)$. Moreover, its index Whittaker transform can be easily deduced using Proposition \ref{lem:iwconv_hatprod}:
\[
\widehat{\mb{e}(a\mu)}(\lambda) = e^{-a} \sum_{k=0}^\infty {a^k \over k!} \widehat{\mu^{*k}}(\lambda) = e^{-a} \sum_{k=0}^\infty {a^k \over k!} \bigl(\widehat{\mu}(\lambda)\bigr)^k = \exp\bigl(a(\widehat{\mu}(\lambda) - 1)\bigr).
\]
Since $\mb{e}((a+b)\mu) = \mb{e}(a\mu) * \mb{e}(b\mu)$, every $*$-compound Poisson measure belongs to $\mathcal{P}_\mathrm{id}$.

\begin{definition} \label{def:iw_gaussian}
A measure $\mu \in \mathcal{P}[0,\infty)$ is called a \emph{$*$-Gaussian measure} if $\mu \in \mathcal{P}_\mathrm{id}$ and
\[
\mu = \mb{e}(a\nu) * \vartheta \quad \bigl(a > 0 ,\, \nu \in \mathcal{P}[0,\infty),\, \vartheta \in \mathcal{P}_\mathrm{id}\bigr) \qquad \implies \qquad \nu = \delta_0.
\]
\end{definition}

\begin{remark}
In the classical case, it is known that if $\mu \in \mathcal{P}(\mathbb{R})$ is a Gaussian measure and $\mu = \mu_1 * \mu_2$ with $\mu_1, \mu_2 \in \mathcal{P}(\mathbb{R})$, then $\mu_1$ and $\mu_2$ are also Gaussian measures. (This is the Lévy-Cramer theorem, cf.\ \cite[\S III.1]{linnikostrovskii1977}.). Conversely, if an infinitely divisible $\mu \in \mathcal{P}(\mathbb{R})$ is such that $\mu = \mb{e}(a\nu) * \vartheta$ implies $\nu = \delta_0$, then the Lévy measure in the classical Lévy-Khintchine formula must be the zero measure, which means that $\mu$ is a Gaussian measure. (See the discussion after Equation (16.8) in \cite{klenke2014}.) This characterization of the classical Gaussian measures makes it natural to define Gaussian measures with respect to the index Whittaker convolution as in Definition \ref{def:iw_gaussian}.
\end{remark}

The following proposition provides a Lévy-Khintchine type representation for $*$-infinitely divisible distributions.

\begin{proposition} \label{prop:iw_levykhin}
The $*$-exponent of a measure $\mu \in \mathcal{P}_\mathrm{id}$ can be represented in the form
\begin{equation} \label{eq:iw_levykhin}
\psi_\mu(\lambda) = \psi_\gamma(\lambda) + \int_{(0,\infty)\!} \bigl( 1 - \bm{W}_{\!\alpha,\Delta_{\lambda\!}}(x) \bigr) \nu(dx)
\end{equation}
where $\nu$ is a $\sigma$-finite measure on $(0,\infty)$ which is finite on $(\eps,\infty)$ for all $\eps > 0$ and such that
\[
\int_{(0,\infty)\!} \bigl( 1 - \bm{W}_{\!\alpha,\Delta_{\lambda\!}}(x) \bigr) \nu(dx) < \infty
\]
and $\gamma$ is a $*$-Gaussian measure with $*$-exponent $\psi_\gamma(\lambda)$. Conversely, each function of the form \eqref{eq:iw_levykhin} is a $*$-exponent of some $\mu \in \mathcal{P}_\mathrm{id}$.
\end{proposition}

\begin{proof}
The proposition follows from the results of Volkovich \cite{volkovich1988} on generalized convolution algebras. Let us verify the corresponding assumptions:
\begin{itemize}[itemsep=1pt]
\item We have seen above that: the index Whittaker transform of any probability measure is uniformly continuous; $\widehat{\delta_0}(\lambda) \equiv 1$; the product property of Proposition \ref{lem:iwconv_hatprod} holds. Hence the index Whittaker convolution satisfies assumptions a), b) and c) of \cite{volkovich1988}. 

\item Proposition \ref{prop:iwmeas_props}(iii)-(iv) shows that the index Whittaker convolution satisfies assumption d) of \cite{volkovich1988}.

\item By the analyticity properties of the Whittaker function, for fixed $x$ the zeros of the function $\lambda \mapsto \bm{W}_{\!\alpha,\Delta_{\lambda\!}}(x)$ are isolated. Moreover, we know that (by virtue of the uniform continuity) there exists no $\mu \in \mathcal{P}[0,\infty)$ with $\widehat{\mu}(\lambda) = 0$ for all $\lambda > 0$. According to Corollary 1 of Volkovich \cite{volkovich1992}, it follows that if $\mathcal{P}_\mathrm{rc}$ is a relatively compact subset of $\mathcal{P}[0,\infty)$ then the set $\mathrm{D}(\mathcal{P}_\mathrm{rc})$ of all divisors (with respect to the $*$-convolution) of the elements of $\mathcal{P}_\mathrm{rc}$ is also a relatively compact subset of $\mathcal{P}[0,\infty)$; this means that the index Whittaker convolution satisfies assumption e) of \cite{volkovich1988}.
\end{itemize}
The conclusion follows from Theorem 4.1 of \cite{volkovich1988}. (By Lemma \ref{lem:iwinfdivis_zeros_idempotent}, the result covers all $*$-infinitely divisible distributions.)
\end{proof}

\section{Lévy processes} \label{chap:iw_levyprocess}

\subsection{Feller-type properties of $*$-convolution semigroups} \label{sec:iw_levyprocess_char}

We start with the definition of a convolution semigroup with respect to the index Whittaker convolution.

\begin{definition}
A family $\{\mu_t\}_{t\geq 0} \subset \mathcal{P}[0,\infty)$ is called a \emph{$*$-convolution semigroup} if it satisfies the conditions
\begin{itemize}[itemsep=-1pt,topsep=3pt]
\item $\mu_s * \mu_t = \mu_{s+t}$ for all $s, t \geq 0$;
\item $\mu_0 = \delta_0$;
\item $\mu_t \warrow \delta_0$ as $t \downarrow 0$.
\end{itemize}
The \emph{(infinitesimal) generator} $A$ of a $*$-convolution semigroup $\{\mu_t\}_{t \geq 0}$ is the operator defined by
\[
Af = \lim_{t \to 0} {1 \over t} (\mathcal{T}^{\mu_t} f - f), \qquad f \in \mathcal{D}_{\!A}
\]
where the domain $\mathcal{D}_{\!A}$ is the set
\[
\mathcal{D}_{\!A} := \Bigl\{ f \in C[0,\infty): \lim_{t \to 0} {1 \over t} (\mathcal{T}^{\mu_t} f - f) \text{ exists in the topology of compact convergence} \Bigr\}.
\]
\end{definition}

\begin{remark}
Similarly to the classical case, the $*$-infinitely divisible distributions are in one-to-one correspondence with the $*$-convolution semigroups (so that the latter also admit a Lévy-Khintchine type representation). Indeed, if $\{\mu_t\}$ is a $*$-convolution semigroup, then $\mu_t$ is (for each $t \geq 0$) a $*$-infinitely divisible distribution; and if $\mu$ is a $*$-infinitely divisible distribution with $*$-exponent $\psi_\mu(\lambda)$, then the semigroup $\{\mu_t\}$ defined by $\widehat{\mu_t}(\lambda) = \exp(-t \, \psi_\mu(\lambda))$ is the unique $*$-convolution semigroup such that $\mu_1 = \mu$. These statements are proved as in the classical case, cf.\ \cite[Section 7]{sato1999}, but replacing the ordinary characteristic function by the index Whittaker transform.
\end{remark}

Unsurprisingly, each $*$-convolution semigroup is associated with a conservative Feller semigroup of operators which commute with the index Whittaker generalized translation:

\begin{proposition}
Let $\{\mu_t\}_{t\geq 0}$ be a $*$-convolution semigroup. Then the family $\{T_t\}_{t \geq 0}$ of convolution operators defined by
\begin{equation} \label{eq:iw_fellersemigr}
\begin{gathered}
T_t: \mathrm{C}_\mathrm{b}[0,\infty) \longrightarrow \mathrm{C}_\mathrm{b}[0,\infty) \\
T_t f = \mathcal{T}^{\mu_{t\!}} f, \quad \text{ where } \quad (\mathcal{T}^{\mu_t{\!}} f)(x) := \int_{[0,\infty)\!} (\mathcal{T}^y f)(x)\, \mu_t(dy)
\end{gathered}
\end{equation}
is a conservative Feller semigroup, i.e., it satisfies the properties
\begin{enumerate}[itemsep=0pt,topsep=4pt]
\item[\textbf{(i)}] $T_t T_s = T_{t+s}$ for all $t, s \geq 0$;
\item[\textbf{(ii)}] $T_t \bigl(\mathrm{C}_0[0,\infty)\bigr) \subset \mathrm{C}_0[0,\infty)$ for all $t \geq 0$;
\item[\textbf{(iii)}] $T_t \mathds{1} = \mathds{1}$ for all $t \geq 0$, and if $f \in \mathrm{C}_\mathrm{b}[0,\infty)$ satisfies $0 \leq f \leq 1$, then $0 \leq T_t f \leq 1$;
\item[\textbf{(iv)}] $\lim_{t \downarrow 0} \|T_t f - f\|_\infty = 0$ for each $f \in \mathrm{C}_0[0,\infty)$.
\end{enumerate}
Moreover, $T_t \mathcal{T}^\nu f = \mathcal{T}^\nu T_t f$ for all $t \geq 0$ and $\nu \in \mathcal{M}_{b}[0,\infty)$ (in particular, $T_t \mathcal{T}^x f = \mathcal{T}^x T_t f$ for $x \geq 0$).
\end{proposition}

\begin{proof}
The proof of (i) is standard (see e.g.\ \cite[Section 4.2]{jewett1975}). Part (ii) follows at once from \eqref{eq:iw_gentransl_contpreserv} and the dominated convergence theorem, and (iii) follows trivially from the fact that the index Whittaker translation operator $\mathcal{T}^y$ is Markovian.

To prove (iv), let $f \in \mathrm{C}_0[0,\infty)$ and $x \geq 0$. From the definition of weak convergence and the fact that $(\mathcal{T}^0 f)(x) = f(x)$ we deduce that
\[
\bigl|(T_t f)(x) - f(x)\bigr| = \biggl| \int_{[0,\infty)\!} \bigl((\mathcal{T}^y f)(x) - f(x)\bigr) \mu_t(dy) \biggr| \xrightarrow[\; t \downarrow 0\;]{} \biggl| \int_{[0,\infty)} \bigl((\mathcal{T}^y f)(x) - f(x)\bigr) \delta_0(dy) \biggr| = 0.
\]
Using a well-known result from the theory of Feller semigroups (e.g.\ \cite[Lemma 1.4]{bottcher2013}), we conclude that (iv) holds.

Lastly, the property \eqref{eq:iw_gentransl_symmxy} extends, by associativity and commutativity of the index convolution, to 
\[
\mathcal{T}^\mu \mathcal{T}^\nu f = \mathcal{T}^\nu \mathcal{T}^\mu f \qquad \text{for all } f \in \mathrm{C}_\mathrm{b}[0,\infty) \text{ and } \mu, \nu \in \mathcal{P}_{\!\mathsmaller\leq}[0,\infty)
\]
and therefore the concluding statement holds.
\end{proof}

\begin{proposition} \label{prop:iwfeller_Lpext}
Let $\{T_t\}$ be a Feller semigroup determined by the $*$-convolution semigroup $\{\mu_t\}_{t \geq 0}$. Then, for each $1 \leq p \leq \infty$, $\bigl\{T_t\restrict{\mathrm{C}_\mathrm{c}[0,\infty)}\bigr\}$ has an extension $\{T_t^{(p)}\}$ which is a strongly continuous Markovian contraction semigroup on $L_p(\mathrm{m})$.
\end{proposition}

\begin{proof}
For $f,g\in \mathrm{C}_\mathrm{c}[0,\infty)$, \eqref{eq:iw_gentransl_symmweig} and Fubini's theorem yield
\begin{equation} \label{eq:iwfeller_Ccsymm}
\begin{aligned}
\int_0^\infty (T_t f)(x) g(x)\, \mathrm{m}(x) dx & = \int_{[0,\infty)\!} \int_0^\infty (\mathcal{T}^y f)(x) g(x)\, \mathrm{m}(x) dx\, \mu_t(dy) \\
& = \int_{[0,\infty)\!} \int_0^\infty f(x) (\mathcal{T}^y g)(x)\, \mathrm{m}(x) dx\, \mu_t(dy) = \int_0^\infty f(x) (T_t g)(x)\, \mathrm{m}(x) dx.
\end{aligned}
\end{equation}
The conclusion now follows from Lemma 1.45 of \cite{bottcher2013}.
\end{proof}

It is worth pointing out that, taking advantage of the correspondence between functions $f \in L_2(\mathrm{m})$ and their index Whittaker transforms (Theorem \ref{thm:iw_transf_L2iso}), the action of the $L_2$-Markov semigroup $\{T_t^{(2)}\}$ can be explicitly written as
\begin{equation} \label{eq:iwfeller_L2char}
\widehat{T_t^{(2)\!} f} = e^{-t\,\psi} \ccdot \widehat{f}, \qquad f \in L_2(\mathrm{m}).
\end{equation}
where $\psi$ is the $*$-exponent of the $*$-convolution semigroup $\{\mu_t\}$ (i.e., of the measure $\mu_1$). Indeed, for $f \in \mathrm{C}_\mathrm{c}[0,\infty)$ we have
\[
\widehat{T_t^{(2)\!} f}(\lambda) = \int_0^\infty \! \bm{W}_{\!\alpha,\Delta_{\lambda\!}}(x) \! \int_{[0,\infty)\!} (\mathcal{T}^y f)(x) \mu_t(dy)\, \mathrm{m}(x) dx = \int_{[0,\infty)\!} \widehat{\mathcal{T}^y f}(\lambda) \, \mu_t(dy) = e^{-t\,\psi(\lambda)} \widehat{f}(\lambda), \quad\; \lambda \geq 0,
\]
(the second and third equalities being obtained by changing the order of integration and using \eqref{eq:iw_gentransl_conn}, respectively) and by continuity the equality extends to all $f \in L_2(\mathrm{m})$.

The index Whittaker transform also allows us to give the following characterization of the generator of the semigroup $\{T_t^{(2)}\}$:

\begin{proposition} \label{prop:iwfeller_L2gen}
Let $\{\mu_t\}$ be a $*$-convolution semigroup with $*$-exponent $\psi$ and let $\{T_t^{(2)}\}$ be the associated Markovian semigroup on $L_2(\mathrm{m})$. Then the infinitesimal generator $(A^{(2)}, \mathcal{D}_{\!A^{(2)}})$ of the semigroup $\{T_t^{(2)}\}$ is the self-adjoint operator given by
\[
\widehat{A^{(2)} f} = -\psi \cdot \widehat{f}, \qquad f \in \mathcal{D}_{\!A^{(2)}}
\]
where
\[
\mathcal{D}_{\!A^{(2)}} = \biggl\{ f \in L_2(\mathrm{m}): \int_{({1 \over 2} - \alpha)^{2\!}}^\infty \bigl|\psi(\lambda)\bigr|^2 \bigl|\widehat{f}(\lambda)\bigr|^2 \rho(\lambda) d\lambda < \infty \biggr\}.
\]
\end{proposition}

\begin{proof}
The proof is very similar to that for the ordinary Fourier transform (see \cite[Theorem 12.16]{bergforst1975}), so we only give a sketch.

From \eqref{eq:iwfeller_Ccsymm} and the usual density argument we see that $\{T_t^{(2)}\}$ is a semigroup of symmetric operators in $L_2(\mathrm{m})$, and therefore (cf.\ \cite[Theorem 4.6]{davies1980}) its generator $(A^{(2)}, \mathcal{D}_{\!A^{(2)}})$ is self-adjoint.

Letting $f \in \mathcal{D}_{\!A^{(2)}}$, so that $L_2\text{-\!}\lim_{t \downarrow 0}{1 \over t} (T_t f - f) = A^{(2)}f \in L_2(\mathrm{m})$, from \eqref{eq:iwfeller_L2char} we get
\[
L_2\text{-\!}\lim_{t \downarrow 0}{1 \over t} \bigl(e^{-t\,\psi} - 1\bigr) \ccdot \widehat{f} = \widehat{A^{(2)}f}
\]
The convergence holds almost everywhere along a sequence $\{t_n\}_{n \in \mathbb{N}}$ such that $t_n \to 0$, so we conclude that $\widehat{A^{(2)}f} = -\psi \cdot \widehat{f} \in L_2\bigl(\Lambda; \rho(\lambda) d\lambda \bigr)$.

Conversely, if we let $f \in L_2(\mathrm{m})$ with $-\psi \ccdot \widehat{f} \in L_2\bigl(\Lambda; \rho(\lambda) d\lambda \bigr)$, then we have
\[
L_2\text{-\!}\lim_{t \downarrow 0} {1 \over t}\bigl(\widehat{T_t f} - \widehat{f} \bigr) = -\psi \ccdot \widehat{f} \in L_2\bigl(\Lambda; \rho(\lambda) d\lambda\bigr)
\]
and the isometry gives that $L_2\text{-\!}\lim_{t \downarrow 0} {1 \over t}\bigl(T_t f - f \bigr) \in L_2(\mathrm{m})$, meaning that $f \in \mathcal{D}_{\!A^{(2)}}$.
\end{proof}

\subsection{$*$-Lévy and $*$-Gaussian processes}

\begin{definition}
Let $\{\mu_t\}_{t\geq 0}$ be a $*$-convolution semigroup. A $[0,\infty)$-valued Markov process $X = \{X_t\}_{t \geq 0}$ is said to be a \emph{$*$-Lévy process} associated with $\{\mu_t\}_{t \geq 0}$ if its transition probabilities are given by
\[
P\bigl[X_t \in B | X_s = x\bigr] = (\mu_{t-s} * \delta_x)(B), \qquad 0 \leq s \leq t,\; x \geq 0,\; B \text{ a Borel subset of } [0,\infty).
\]
\end{definition}

In other words, a $*$-Lévy process is a Feller process associated with the Feller semigroup defined in \eqref{eq:iw_fellersemigr}. Consequently, the general connection between Feller semigroups and Feller processes (see e.g.\ \cite[Section 1.2]{bottcher2013}) assures that for each (initial) distribution $\nu \in \mathcal{P}[0,\infty)$ and $*$-convolution semigroup $\{\mu_t\}_{t\geq 0}$ there exists a $*$-Lévy process $X$ associated with $\{\mu_t\}_{t \geq 0}$ and such that $P_{X_0} = \nu$. Being a Feller process, any $*$-Lévy process is stochastically continuous (i.e., $X_s \to X_t$ in probability as $s \to t$, for each $t \geq 0$) and has a càdlàg modification (i.e., there exists a $*$-Lévy process $\{\widetilde{X}_t\}$ with a.s.\ right-continuous paths and satisfying $P\bigl[X_t = \widetilde{X}_t \bigr] = 1$ for all $t \geq 0$).

\begin{example} \label{exam:iw_diffusion}
Consider the one-dimensional diffusion process $Y = \{Y_t\}_{t \geq 0}$ generated by the differential operator \eqref{eq:iw_Lop}, i.e.\ the solution of the stochastic differential equation
\[
dY_t = {1 \over 4}\bigl(Y_t^{-1} + (3-4\alpha)Y_t\bigr) dt + 2^{-{1 \over 2}} Y_t \, dW_t
\]
where $\{W_t\}_{t \geq 0}$ is a standard Brownian motion. This diffusion process, which we call the \emph{index Whittaker diffusion}, is a $*$-Lévy process. Indeed, it follows from \cite[Theorem 3.6 and Example 3.8]{sousayakubovich2017} that the process $Y$ has transition probabilities given by
\[
p_{t,x}(dy) \equiv P\bigl[Y_t \in dy | Y_0 = x\bigr] = \int_{({1 \over 2} - \alpha)^{2\!}}^\infty \!  e^{-t\lambda} \, \bm{W}_{\!\alpha, \Delta_{\lambda\!}}(x)  \bm{W}_{\!\alpha, \Delta_{\lambda\!}}(y) \, \rho(\lambda) d\lambda \, \mathrm{m}(y) dy, \qquad t > 0, \; x \geq 0
\]
Computing the index Whittaker transform, we get
\[
\widehat{p_{t,x}}(\lambda) = \int_0^\infty \! \int_{({1 \over 2} - \alpha)^{2\!}}^\infty \!  e^{-t\lambda} \bm{W}_{\!\alpha, \Delta_{\lambda\!}}(x)  \bm{W}_{\!\alpha, \Delta_{\lambda\!}}(y) \, \rho(\lambda) d\lambda \: \bm{W}_{\!\alpha, \Delta_{\lambda\!}}(y)\, \mathrm{m}(y) dy = e^{-t\lambda} \bm{W}_{\!\alpha, \Delta_{\lambda\!}}(x),\, \quad\;\; t > 0, \, \lambda \geq 0
\]
where the last equality follows from Theorem \ref{thm:iw_transf_L2iso} (noting that the integral $\int_0^\infty \bm{W}_{\!\alpha, \Delta_{\lambda\!}}(y) \, p_{t,x}(dy)$ converges absolutely). This shows that $p_{t,x} = \mu_t * \delta_x$ where $\mu_t = p_{t,0}$. Given that $\widehat{\mu_t}(\lambda) = e^{-t\lambda}$, it is clear that $\{\mu_t\}$ is a $*$-convolution semigroup and, therefore, $Y$ is a $*$-Lévy process.

We observe that (as emphasized in the Introduction) the index Whittaker diffusion can also be defined as $Y_t = \smash{\sqrt{{1 \over 2} V_t}}$, where $V = \{V_t\}_{t \geq 0}$ is the Shiryaev process \cite{peskir2006}. The latter is defined as the solution of the stochastic differential equation $dV_t = (1 + 2(1-\alpha) V_t) dt + 2^{1 \over 2} V_t \, dW_t$, and its transition probability density is governed by the Kolmogorov-Shiryaev equation ${\partial u(t,x) \over \partial t} = -{\partial \over \partial x} \bigl[(1+2(1-\alpha)x)u(t,x)\bigr] + {\partial^2 \over \partial x^2} \bigl[x^2 u(t,x)\bigr]$.
\end{example}

Some equivalent characterizations of $*$-Lévy processes are given in the next proposition. For a càdlàg process $X = \{X_t\}_{t \geq 0}$ and a generator $A$ of a $*$-convolution semigroup $\{\mu_t\}_{t \geq 0}$, we introduce the notation $Z\text{\large\mathstrut}_{X}^{A,f} = \bigl\{Z\text{\large\mathstrut}_{X,t}^{A,f}\bigr\}_{t \geq 0}$, where
\begin{equation} \label{eq:iw_levy_Zdef}
Z\text{\large\mathstrut}_{X,t}^{A,f} := f(X_t) - f(X_0) - \int_0^t A(g)(X_s)\, ds \qquad\quad (f \in \mathcal{D}_{\!A}).
\end{equation}

\begin{proposition} \label{prop:iw_levy_equivchar}
Let $\{\mu_t\}$ be a $*$-convolution semigroup with $*$-exponent $\psi$ and let $(A,\mathcal{D}_{\!A})$ be its generator. Let $X$ be a $[0,\infty)$-valued càdlàg Markov process. The following assertions are equivalent:
\begin{enumerate} [itemsep=0pt,topsep=4pt]
\item[\textbf{(i)}] $X$ is a $*$-Lévy process associated with $\{\mu_t\}$;
\item[\textbf{(ii)}] $\{e^{t\psi(\lambda)} \bm{W}_{\!\alpha,\Delta_{\lambda\!}}(X_t)\}_{t \geq 0}$ is a martingale for each $\lambda \geq 0$;
\item[\textbf{(iii)}] $\bigl\{ \bm{W}_{\!\alpha,\Delta_{\lambda\!}}(X_t) - \bm{W}_{\!\alpha,\Delta_{\lambda\!}}(X_0) + \psi(\lambda) \int_0^t \bm{W}_{\!\alpha,\Delta_{\lambda\!}}(X_s)\, ds \bigr\}_{t \geq 0}$ is a martingale for each $\lambda \geq 0$;
\item[\textbf{(iv)}] $Z\text{\large\mathstrut}_{X}^{A,\raisebox{.5pt}{\scalebox{0.75}{\footnotesize$\bm{W}_{\!\alpha,\Delta_{\lambda\!}}(\cdot)$}}}$ is a martingale for each $\lambda \geq 0$;
\item[\textbf{(v)}] $Z\text{\large\mathstrut}_{X}^{A,f}$ is a martingale for each $f \in \mathcal{D}_{\!A} \cap \mathrm{C}_0[0,\infty)$.
\end{enumerate}
\end{proposition}

\begin{proof}
The proof is identical to that of \cite[Theorem 3.4]{rentzschvoit2000}.
\end{proof}

A $*$-convolution semigroup $\{\mu_t\}_{t \geq 0}$ such that $\mu_1$ is a $*$-Gaussian measure will be called a \emph{$*$-Gaussian convolution semigroup}, and a $*$-Lévy process associated with a $*$-Gaussian convolution semigroup is said to be a \emph{$*$-Gaussian process}. An alternative characterization of $*$-Gaussian convolution semigroups (which in particular implies that any $*$-Gaussian convolution semigroup is fully composed of $*$-Gaussian measures) is given in the next lemma.

\begin{lemma} \label{lem:iw_gaussian_equiv}
Let $\mu \in \mathcal{P}_\mathrm{id}$ and let $\{\mu_t\}$ be the $*$-convolution semigroup $\{\mu_t\}$ such that $\mu_1 = \mu$. Then, the following conditions are equivalent:
\begin{enumerate}[itemsep=0pt,topsep=4pt]
\item[\textbf{(i)}] $\mu$ is a $*$-Gaussian measure;
\item[\textbf{(ii)}] $\lim_{t \downarrow 0} {1 \over t} \mu_t[\eps,\infty) = 0 \;$ for every $\eps > 0$;
\item[\textbf{(iii)}] $\lim_{t \downarrow 0} {1 \over t} (\mu_t*\delta_x)\bigl([0,\infty) \setminus (x-\eps,x+\eps)\bigr) = 0 \;$ for every $x \geq 0$ and $\eps > 0$.
\end{enumerate}
\end{lemma}

\begin{proof}
\textbf{\emph{(i)$\!\iff\!$(ii):\;}} Let $\{t_n\}_{n \in \mathbb{N}}$ be any sequence of positive numbers such that $t_n \to 0$ as $n \to \infty$. We know that $\widehat{\mu_t}(\lambda) = \bigl(\widehat{\mu}(\lambda)\bigr)^{\!t}$ for all $t > 0$ and that $\widehat{\mu}(\lambda) > 0$ for all $\lambda$ (see Lemma \ref{lem:iwinfdivis_zeros_idempotent}), so we deduce
\begin{equation} \label{eq:iw_gauss_equivdef_pf1}
\lim_{n \to \infty} \bigr[ \mb{e}\bigl(\tfrac{1}{t_n} \mu_{t_{n\!}}\bigr) \bigr]^{\!\mathlarger\wedge} (\lambda) = \lim_{n \to \infty} \exp\biggl[ {1 \over t_n}\bigl(\widehat{\mu_{t_n}}(\lambda) - 1\bigr) \biggr] = \widehat{\mu}(\lambda), \qquad \lambda > 0.
\end{equation}
Hence, by Proposition \ref{prop:iwmeas_props}(iv), $\mb{e}\bigl(\tfrac{1}{t_n} \mu_{t_n}\bigr) \warrow \mu$ as $n \to \infty$. Notice that \eqref{eq:iw_gauss_equivdef_pf1} shows, in particular, that $\limsup_{n \to \infty} {1 \over t_n}\bigl(1 - \widehat{\mu_{t_n}}(\lambda)\bigr) < \infty$. Using Theorem 3.1 of Volkovich \cite{volkovich1988}, we conclude that $\mu$ is Gaussian if and only if
\[
\lim_{n \to \infty} {1 \over t_n} \mu_{t_n}[\eps,\infty) = 0 \qquad \text{for every } \eps > 0.
\]

\textbf{\emph{(ii)$\!\iff\!$(iii):\;}} To prove the nontrivial direction, assume that (ii) holds, and fix $x, \eps > 0$ with $0 < x < \eps$. Write $E_\eps = [0,\infty) \setminus (x-\eps, x+\eps)$, and let $\mathds{1}_\eps$ denote its indicator function. We start the proof by establishing an upper bound for the function $(\mathcal{T}^x \mathds{1}_\eps)(y)$, with $y>0$ small. Using the estimate \eqref{eq:bW_prodform_kernbound}, together with the inequality
\[
\begin{aligned}
{(x^2-\xi^2)^2 \over 8(xy\xi)^2} & \geq {({\eps \over x}+{\eps \over \xi})^2 \over (4y)^2} + {(x^2-\xi^2)^2 \over (4\delta x\xi)^2} \\[2pt]
& \geq {\eps^2 \over (4xy)^2} + {(x^2-\xi^2)^2 \over (4\delta x\xi)^2}
\end{aligned} \qquad\qquad (y \leq \delta, \; \xi \in E_\eps)
\]
it is easily seen that
\[
q(x,y,\xi) \leq C_1\, y^{-1} (\xi+\xi^{-1}) \exp\biggl( {1 \over 4\xi^2} - {\eps^2 \over (4xy)^2} - {(x^2-\xi^2)^2 \over (4\delta x\xi)^2} \biggr), \qquad\qquad y \leq \delta, \; \xi \in E_\eps
\]
where the constant $C_1 > 0$ depends only on $x$, $\delta$ and $\eps$. Consequently,
\begin{align}
\nonumber(\mathcal{T}^x \mathds{1}_\eps)(y) & = \int_{E_\eps} q(x,y,\xi) \mathrm{m}(\xi) d\xi \\
& \nonumber \leq C_1\, y^{-1} \exp\biggl(- {\eps^2 \over (4xy)^2} \biggr) \! \int_0^\infty \xi^{-4\alpha} (1+\xi^2) \exp\biggl(- {1 \over 4\xi^2} - {(x^2-\xi^2)^2 \over (4\delta x\xi)^2} \biggr) d\xi \\
\label{eq:iw_gaussian_char_translineq} & \leq C_2\, y^{-1} \exp\biggl(- {\eps^2 \over (4xy)^2} \biggr), \qquad\quad y \leq \delta 
\end{align}
the convergence of the integral justifying that the last inequality holds for a possibly larger constant $C_2$.

Let $\lambda > 0$ be arbitrary. If $\delta > 0$ is sufficiently small, then from \eqref{eq:iwmeas_infdiv_bound_pf0} and \eqref{eq:iw_gaussian_char_translineq} it follows that
\[
{(\mathcal{T}^x \mathds{1}_\eps)(y) \over 1 - \bm{W}_{\!\alpha,\Delta_{\lambda\!\!}}(y)} \leq {C_2 \over \lambda} \, y^{-3} \exp\biggl(- {\eps^2 \over (4xy)^2} \biggr), \qquad y \leq \delta
\]
and therefore there exists $\delta' < \delta$ such that ${(\mathcal{T}^x \mathds{1}_\eps)(y) \leq 1 - \bm{W}_{\!\alpha,\Delta_{\lambda\!\!}}(y)}$ for all $y \in [0,\delta')$. We then estimate
\begin{align*}
{1 \over t} (\mu_t*\delta_x)(E_\eps) & = {1 \over t} \int_{[0,\infty)\!} (\mathcal{T}^x \mathds{1}_\eps)(y) \mu_t(dy) \\
& \leq {1 \over t} \int_{[0,\delta')\!} \bigl( 1 - \bm{W}_{\!\alpha,\Delta_{\lambda\!\!}}(y) \bigr) \mu_t(dy) + {1 \over t} \mu_t [\delta',\infty) \\
& \leq {1 \over t} \int_{[0,\infty)\!} \bigl( 1 - \bm{W}_{\!\alpha,\Delta_{\lambda\!\!}}(y) \bigr) \mu_t(dy) + {1 \over t} \mu_t [\delta',\infty) \\
& = {1 \over t} \bigl( 1-\widehat{\mu_t}(\lambda) \bigr) + {1 \over t} \mu_t [\delta',\infty).
\end{align*}
Since we are assuming that (ii) holds and we know from \eqref{eq:iw_gauss_equivdef_pf1} that $\lim_{t \downarrow 0} {1 \over t} \bigl( 1-\widehat{\mu_t}(\lambda) \bigr) = -\log\widehat{\mu}(\lambda)$, the above inequality gives
\[
\limsup_{t \downarrow 0} {1 \over t} (\mu_t*\delta_x)(E_\eps) \leq -\log\widehat{\mu}(\lambda).
\]
By the properties of the index Whittaker transform, the right-hand side is continuous and vanishes for $\lambda = 0$, so from the arbitrariness of $\lambda$ we see that $\lim_{t \downarrow 0} {1 \over t} (\mu_t*\delta_x)(E_\eps) = 0$, as desired.
\end{proof}

Going back to Example \ref{exam:iw_diffusion}, it follows from \cite[Theorem 4.5]{mckean1956} that the index Whittaker diffusion is such that $\lim_{t\downarrow 0} {1 \over t} p_{t,x}\bigl([0,\ \infty)\setminus(x-\eps,x+\eps)\bigr) = 0$ for any $x \geq 0$ and $\eps > 0$, meaning that the index Whittaker diffusion is a $*$-Gaussian process. It turns out that, as a consequence of the previous lemma, any other $*$-Gaussian process is also a one-dimensional diffusion:

\begin{corollary}
Let $X = \{X_t\}_{t \geq 0}$ be a $*$-Gaussian process. Then:
\begin{enumerate}[itemsep=0pt,topsep=4pt]
\item[\textbf{(i)}] $X$ has a modification whose paths are a.s.\ continuous;
\item[\textbf{(ii)}] Let $A$ be the generator of the $*$-Gaussian convolution semigroup associated with $X$. Then $A$ is a local operator, i.e., $(Af)(x) = (Ag)(x)$ whenever $f, g \in \mathcal{D}_{\! A} \cap \mathrm{C}_0[0,\infty)$ and $f=g$ on some neighborhood of $x \geq 0$.
\end{enumerate} 
\end{corollary}

\begin{proof}
According to Lemma \ref{lem:iw_gaussian_equiv}, the $*$-Gaussian convolution semigroup $\{\mu_t\}$ associated with $X$ satisfies $\lim_{t \downarrow 0} {1 \over t} (\mu_t*\delta_x)\bigl([0,\infty) \setminus (x-\eps,x+\eps)\bigr) = 0 \;$ for every $x \geq 0$ and $\eps > 0$. Applying Corollary 3 of \cite{schilling1994}, we conclude that the càdlàg modification of $X$ has a.s.\ continuous sample paths. The locality of the generator then follows from Theorem 1.40 of \cite{bottcher2013}.

(The two results used in the proof are stated for processes with state space $\mathbb{R}^d$, but we can apply them to the $\mathbb{R}$-valued process $\widetilde{X} = \{\widetilde{X}_t\}_{t \geq 0}$ which is the extension of $X$ obtained by setting $\widetilde{X}_t(\omega) = x$ whenever the initial distribution is $\nu = \delta_x$, $x < 0$.)
\end{proof}

\subsection{Moment functions} \label{sec:iw_moment}

Moment functions for the index Whittaker convolution are functions having the same additivity property which is satisfied by the monomials under the classical convolution:

\begin{definition}
The sequence of functions $\{\varphi_k\}_{k=1,\ldots,n}$ is said to be a \emph{$*$-moment sequence (of length $n$)} if $\varphi_k \in \mathcal{E}$ for $k=1,\ldots,n$ (cf.\ \eqref{eq:iw_transl_Espace_def}) and
\begin{equation} \label{eq:iw_moment_def}
(\mathcal{T}^y \varphi_k)(x) = \sum_{j=0}^k \binom{k}{j} \varphi_j(x) \varphi_{k-j}(y) \qquad\quad (k=1,\ldots,n; \; x, y \geq 0)
\end{equation}
where $\varphi_0(x) := 1$ ($x \geq 0$).
\end{definition}

It is worth recalling that for $x,y > 0$ the left-hand side of \eqref{eq:iw_moment_def} is given by $\int_0^\infty \! \varphi_k(\xi) q(x,y,\xi)\, \mathrm{m}(\xi) d\xi$, the integral being absolutely convergent. This actually implies that $*$-moment functions are necessarily smooth:

\begin{lemma} \label{lem:iw_transl_Cinfty}
If $\{\varphi_k\}_{k=1,\ldots,n}$ is a $*$-moment sequence, then $\varphi_k \in \mathrm{C}^\infty(0,\infty)$ for all $k$.
\end{lemma}

\begin{proof}
Let $M > 0$ and $1 \leq k \leq n$. Let $f \in  \mathrm{C}_\mathrm{c}^\infty(2M,3M)$ be such that $\int_{2M}^{3M} f(x) \, \mathrm{m}(x) dx = 1$, and set $f(x) = 0$ for $x \notin (2M,3M)$. Then
\[
\sum_{j=0}^k \binom{k}{j} \varphi_j(y) \! \int_0^\infty \! \varphi_{k-j}(x) f(x)\, \mathrm{m}(x) dx \, = \int_0^\infty (\mathcal{T}^y \varphi_k)(x) \, f(x)\, \mathrm{m}(x) dx \, = \int_0^\infty \! \varphi_k(x) \, (\mathcal{T}^y f)(x)\, \mathrm{m}(x) dx.
\]
where the second equality follows from the identity \eqref{eq:iw_gentransl_symmweig}, which is easily seen to hold also for $f \in \mathrm{C}_\mathrm{c}^\infty(0,\infty)$ and $g \in \mathcal{E}$. Hence if we prove that the right-hand side is an infinitely differentiable function of $0 < y < M$, then by induction it follows that each $\varphi_k \in \mathrm{C}^\infty(0,M)$ and, by arbitrariness of $M$, $\varphi_k \in \mathrm{C}^\infty(0,\infty)$.

By \eqref{eq:bW_prodform_kernbound}, we have
\begin{align*}
(\mathcal{T}^y f)(x) & \leq {C_1 \|f\|_\infty \over xy} \exp\biggl( {1 \over 4y^2} \biggr) \int_{2M}^{3M} \xi^{-4\alpha} (x^2+y^2+\xi^2)^{2\alpha} \exp\biggl( - {1 \over 4\xi^2} - {x^2 \over 8(y\xi)^2} - {(y^2-\xi^2)^2 \over 8(xy\xi)^2} \biggr) d\xi \\
& \leq C_2 \|f\|_{\infty\,} y^{-1} (1+x^{4\alpha}) \exp\biggl( {1 \over 4y^2} - {x^2 \over 8(3M^2)^2} - {1 \over 8x^2}\biggr) 
\end{align*}
where $C_1$ and $C_2$ are constants depending only on $M$. Since $\varphi_k \in \mathcal{E}$, we find that
\begin{equation} \label{eq:iw_moment_Cinfty_pf1}
\varphi_k(x) \, (\mathcal{T}^y f)(x) \, \mathrm{m}(x) \leq C \, y^{-1} x \, (1 + x^{-4\alpha}) \exp\biggl( {1 \over 4y^2} + b_2(x^\beta + x^{-\beta}) - {x^2 \over 8(3M^2)^2} - {1 \over 8x^2}\biggr)
\end{equation}
where $C > 0$, $b_2 \geq 0$ and $0 \leq \beta < 2$ do not depend on $y$. Denoting the right-hand side of \eqref{eq:iw_moment_Cinfty_pf1} by $J(x,y)$, it is easily seen that the integral $\int_0^\infty J(x,y) dx$ converges locally uniformly and, therefore, $\int_0^\infty \! \varphi_k(x) \, (\mathcal{T}^y f)(x)\, \mathrm{m}(x) dx$ is a continuous function of $0 < y < M$. Using the identity
\[
{\partial q_\alpha(x,y,\xi) \over \partial y} = {x^2 + \xi^2 - y^2 \over 2y^3 (x \xi)^2} \, q_{\alpha + {1 \over 2}}(x,y,\xi) - \bigl( y^{-3} + (1-2\alpha) y^{-1} \bigr) \, q_\alpha(x,y,\xi)
\]
(which follows from \cite[Equation 8.2(16)]{erdelyiII1953}) and similar arguments, one can derive an upper bound for the derivatives ${\partial^p \over \partial y^p} (\mathcal{T}^y f)(x)\,$ ($p=1,2,\ldots$) and then deduce that $\int_0^\infty \! \varphi_k(x) \, (\mathcal{T}^y f)(x)\, \mathrm{m}(x) dx$ is $p$ times continuously differentiable.
\end{proof}

\begin{proposition} \label{prop:iw_moment_oderep}
If $\{\varphi_k\}_{k=1,\ldots,n}$ is a $*$-moment sequence, then there exist $\lambda_1, \ldots, \lambda_n \in \mathbb{R}$ such that
\[
\mathcal{L}\varphi_k(x) = \sum_{j=1}^k \binom{k}{j} \lambda_j \varphi_{k-j}(x), \qquad \varphi_k(0) = 0, \qquad \varphi_k'(0) = 0 \qquad\quad (k = 1, \ldots, n),
\]
where $\mathcal{L}$ is the differential operator \eqref{eq:iw_Lop}.
\end{proposition}

\begin{proof}
First we will show that $\varphi_k \in \mathrm{C}^\infty(0,\infty) \cap \mathrm{C}^1[0,\infty)$ with $\varphi_k(0) = \varphi_k'(0) = 0$. We know that $\varphi_k \in \mathcal{E} \cap \mathrm{C}^\infty(0,\infty)$, so from Lemma \ref{lem:iw_transl_Espace_limits} it follows that for fixed $y > 0$ we have $\lim_{x \to 0} (\mathcal{T}^y \varphi_k)(x) = \varphi_k(y)$ and $\lim_{x \to 0} {\partial \over \partial x}(\mathcal{T}^y \varphi_k)(x) = 0$. If we rewrite \eqref{eq:iw_moment_def} as
\begin{equation} \label{eq:iw_moment_rewrit}
\varphi_k(x) = (\mathcal{T}^y \varphi_k)(x) - \sum_{j=0}^{k-1} \binom{k}{j} \varphi_j(x) \varphi_{k-j}(y)
\end{equation}
and let $x \to 0$ on the right-hand side, we deduce (by induction on $k$) that $\lim_{x \to 0} \varphi_k(x) = 0$ for all $k$. After differentiating both sides of \eqref{eq:iw_moment_rewrit}, we similarly find that $\lim_{x \to 0} \varphi_k'(x) = 0$ for each $k$.

We now prove that $\varphi_k$ satisfies the given ordinary differential equation, omitting the details which are similar to the proof of \cite[Theorem 4.5]{szekelyhidi2013}. Using the same kind of arguments as in the proof of \cite[Lemma 4.15]{sousayakubovich2017}, we see that $\mathcal{L}_x q(x,y,\xi) = \mathcal{L}_y q(x,y,\xi)$ (where $\mathcal{L}_x$ and $\mathcal{L}_y$ denote the differential operator \eqref{eq:iw_Lop} acting on the variable $x$ and $y$ respectively). Moreover, the identity \eqref{eq:bW_prodform_kernderiv} allows us to verify that the integral defining $(\mathcal{T}^y f)(x)$ can be differentiated under the integral sign. Therefore, the right-hand side of \eqref{eq:iw_moment_def} is, for each $k$, a solution of $\mathcal{L}_x u = \mathcal{L}_y u$, i.e.
\[
\sum_{j=0}^k \binom{k}{j} (\mathcal{L}\varphi_j)(x) \, \varphi_{k-j}(y) = \sum_{j=0}^k \binom{k}{j} \varphi_j(x) \, (\mathcal{L}\varphi_{k-j})(y).
\]
We can assume by induction that $\mathcal{L}\varphi_\ell(x) = \sum_{j=1}^\ell \binom{\ell}{j} \lambda_j \varphi_{\ell-j}(x)$ for $\ell = 1, \ldots, k-1$. Using the induction hypothesis and rearranging the terms in a suitable way, we find that
\[
(\mathcal{L} \varphi_k)(x) - \sum_{j=1}^{k-1} \lambda_j \varphi_{k-j}(x) = (\mathcal{L} \varphi_k)(y) - \sum_{j=1}^{k-1} \lambda_j \varphi_{k-j}(y), \qquad \text{ for all } x, y > 0
\]
and, consequently,
\[
(\mathcal{L} \varphi_k)(x) - \sum_{j=1}^{k-1} \lambda_j \varphi_{k-j}(x) = \lambda_k
\]
for some $\lambda_k \in \mathbb{R}$, finishing the proof.
\end{proof}

The functions $\widetilde{\varphi}_k$ defined as the unique solution of
\begin{equation} \label{eq:iw_moment_canonicalode}
\mathcal{L}\widetilde{\varphi}_k(x) = - k (1-2\alpha) \widetilde{\varphi}_{k-1}(x) - k (k-1) \widetilde{\varphi}_{k-2}(x), \qquad \widetilde{\varphi}_k(0) = 0, \qquad \widetilde{\varphi}_k'(0) = 0 \qquad\quad (k \in \mathbb{N})
\end{equation}
(where $\widetilde{\varphi}_{-1}(x) := 0$ and $\widetilde{\varphi}_0(x) := 1$) are said to be the \emph{canonical $*$-moment functions}. Of course, $\{\widetilde{\varphi}_k\}_{k=1,\ldots,n}$ is indeed a $*$-moment sequence (the reader can verify this by applying the representations below). By integration of the differential equation, we find the explicit recursive expression
\begin{equation} \label{eq:iw_moment_canonicalintrep}
\widetilde{\varphi}_k(x) = 4k \! \int_0^x {1 \over y^{2\,} \mathrm{m}(y)} \! \int_0^y \mathrm{m}(\xi) \, \bigl[(1-2\alpha) \widetilde{\varphi}_{k-1}(\xi) + (k-1) \widetilde{\varphi}_{k-2}(\xi)\bigr]\, d\xi\, dy, \qquad\quad k \in \mathbb{N}.
\end{equation}
Moreover, as a consequence of the uniqueness of solution for \eqref{eq:iw_moment_canonicalode} and the Laplace representation \eqref{eq:bW_laplacerep}, the canonical moment functions can also be represented as
\[
\widetilde{\varphi}_k(x) = {\partial^k \over \partial \sigma^k}\raisebox{-.5ex}{\Big|}_{\sigma = {1 \over 2} - \alpha\!\!} \bm{W}_{\!\alpha,\sigma}(x) \, = \int_{-\infty\!}^\infty \biggl[{\partial^k (e^{\sigma s\!}) \over \partial \sigma^k}\raisebox{-.5ex}{\Big|}_{\sigma = {1 \over 2} - \alpha}\biggr] \eta_x(s) ds \, = \int_{-\infty\!}^\infty s^k e^{({1 \over 2} - \alpha) s} \eta_x(s) ds.
\]
The first (canonical) moment function can be written in closed form:

\begin{proposition}
We have
\begin{equation} \label{eq:iw_moment_closedform}
\widetilde{\varphi}_1(x) = {1 \over \Gamma(1-2\alpha)} G_{23}^{31}\biggl( {1 \over 2x^2}\Bigm| \begin{matrix}
0, 1 \\[-1pt] \, 0, 0, 1-2\alpha
\end{matrix} \biggr)
\end{equation}
where $G_{pq}^{mn}\bigl(z\,| \begin{smallmatrix}
a_1, \ldots, a_p \\ b_1, \ldots, b_q
\end{smallmatrix}\bigr)$ denotes the Meijer-G function \cite[Section 8.2]{prudnikovIII1989}. In the particular case $\alpha = 0$, we have $\widetilde{\varphi}_1(x) = e^{1 \over 2x^2} \Gamma(0,{1 \over 2x^2})$, where $\Gamma(a,z)$ is the incomplete Gamma function \cite[Chapter IX]{erdelyiII1953}.
\end{proposition}

\begin{proof}
We know from \eqref{eq:iw_moment_canonicalintrep} that $\widetilde{\varphi}_1(x) = 4(1-2\alpha) \int_0^x {1 \over y^{2\,} \mathrm{m}(y)} \! \int_0^y \mathrm{m}(\xi) d\xi\, dy$. Consequently,
\begin{align*}
\widetilde{\varphi}_1(x) & = (1-2\alpha) \int_{1 \over 2x^2}^\infty v^{-2\alpha} e^v \int_v^\infty w^{2\alpha-2} e^{-w} dw\, dv \\
& = (1-2\alpha) \int_{1 \over 2x^2}^\infty v^{-2\alpha} e^v \Gamma(-1+2\alpha,v) dv \\
& = {1 \over \Gamma(1-2\alpha)} \int_{1 \over 2x^2}^\infty \mskip 0.6\thinmuskip G_{12}^{21}\biggl( v \Bigm| \begin{matrix}
-1 \\[-1pt] -2\alpha, -1
\end{matrix} \biggr) dv \\
& = {1 \over \Gamma(1-2\alpha)} G_{23}^{31}\biggl( {1 \over 2x^2}\Bigm| \begin{matrix}
0, 1 \\[-1pt] \, 0, 0, 1-2\alpha
\end{matrix} \biggr)
\end{align*}
where the first equality is obtained via a change of variables, the second equality follows from the definition of the incomplete Gamma function, the third step is due to \cite[Relations 8.2.2.15 and 8.4.16.13]{prudnikovIII1989} and the final step applies \cite[Equation 5.6.4(6)]{luke1969}. The result for $\alpha = 0$ follows from the identity $G_{23}^{31}\bigl( {1 \over 2x^2} | \begin{smallmatrix}
0, 1 \\[-1pt] \, 0, 0, 1
\end{smallmatrix} \bigr) = e^{1 \over 2x^2} \Gamma(0,{1 \over 2x^2})$, cf.\ \cite[Relations 8.2.2.9 and 8.4.16.13]{prudnikovIII1989}.
\end{proof}

Actually, the right-hand side of \eqref{eq:iw_moment_closedform} can be written (for $\alpha < {1 \over 2}$) as a sum of simpler special functions. Such representation can be obtained by applying \cite[Equation (A13)]{barrabes1999}.

Returning to moment functions of general order, it is clear from the explicit representation \eqref{eq:iw_moment_canonicalintrep} that $\widetilde{\varphi}_k(x) > 0$ for all $x > 0$ and $k \in \mathbb{N}$. We note that $\widetilde{\varphi}_2 \geq \widetilde{\varphi}_1^2$ (by Jensen's inequality applied to $\widetilde{\varphi}_k(x) = \int_{-\infty\!}^\infty s^k e^{({1 \over 2} - \alpha) s} \eta_x(s) ds$) and that the Taylor expansions of the first two moment functions as $x \to 0$ are
\begin{equation} \label{eq:iw_moment_canontaylor}
\widetilde{\varphi}_1(x) = 2(1-2\alpha) x^2 - 4(1-2\alpha)(1-\alpha) x^4 + o(x^4), \qquad\; \widetilde{\varphi}_2(x) = 4 x^2 - 4 (1+2\alpha-4\alpha^2) x^4 + o(x^4)
\end{equation}
(this follows from \eqref{eq:bW_limderivatives}, taking into account that ${d^j \over dx^j} \widetilde{\varphi}_k(x) = {\partial^k \over \partial \sigma^k}\smash{\raisebox{-.25ex}{\big|}_{\sigma = {1 \over 2} - \alpha}} {d^j \over dx^j} \bm{W}_{\!\alpha,\sigma}(x)$). Concerning the growth of the moment functions as $x \to \infty$, we have:

\begin{proposition}
Let $\eps > 0$. For each $k \in \mathbb{N}$, $\widetilde{\varphi}_k(x) = O(x^\eps)$ as $x \to \infty$.
\end{proposition}

\begin{proof}
Due to \eqref{eq:bW_liminfty}, it suffices to prove that $\widetilde{\varphi}_k(x) = O\bigl(\bm{W}_{\alpha,{1 \over 2} - \alpha + {\eps \over 2}}(x)\bigr)$ as $x \to \infty$ for each $k \in \mathbb{N}$. This is trivial for $k = 0$ since $\widetilde{\varphi}_0 \equiv 1 = O(x^{\eps}) = O\bigl(\bm{W}_{\alpha,{1 \over 2} - \alpha + {\eps \over 2}}(x)\bigr)$. By induction, suppose that $\widetilde{\varphi}_j(x) = O\bigl(\bm{W}_{\alpha,{1 \over 2} - \alpha + {\eps \over 2}}(x)\bigr)$ for $j = 0, \ldots, k-1$. This implies that $\widetilde{\varphi}_j(x) \leq C \ccdot \bm{W}_{\alpha,{1 \over 2} - \alpha + {\eps \over 2}}(x)$ for all $x \geq 0$ and $j = 0, \ldots, k-1$ (where $C > 0$ does not depend on $x$). Recalling \eqref{eq:bW_integraleq} and \eqref{eq:iw_moment_canonicalintrep}, we find
\[
\widetilde{\varphi}_k(x) \leq C \! \int_0^x {1 \over y^{2\,} \mathrm{m}(y)} \! \int_0^y \mathrm{m}(\xi) \, \bm{W}_{\alpha,{1 \over 2} - \alpha + {\eps \over 2}}(\xi) d\xi\, dy = C \ccdot {1 \over \eps(2-4\alpha+\eps)} \bigl( \bm{W}_{\alpha,{1 \over 2} - \alpha + {\eps \over 2}}(x) - 1 \bigr)
\]
and therefore $\widetilde{\varphi}_k(x) = O\bigl(\bm{W}_{\alpha,{1 \over 2} - \alpha + {\eps \over 2}}(x)\bigr)$, proving the proposition.
\end{proof}

The previous proposition shows that the \emph{modified moments} $\mathbb{E}[\widetilde{\varphi}_k(X)]$ will only diverge if the tails of the random variable $X$ are very heavy. The next result shows that the modified moments can be computed via the index Whittaker transform:

\begin{proposition} \label{prop:iw_moment_finitenesschar}
Let $\mu \in \mathcal{P}[0,\infty)$ and $k \in \mathbb{N}$. The following assertions are equivalent:
\begin{enumerate} [itemsep=0pt,topsep=4pt]
\item[\textbf{(i)}] $\int_{[0,\infty)} \widetilde{\varphi}_k(x) \mu(dx) < \infty$;
\item[\textbf{(ii)}] $\sigma \mapsto \int_{[0,\infty)} \! \bm{W}_{\!\alpha,\sigma}(x) \mu(dx)$ is $k$ times differentiable on $[0,{1 \over 2} - \alpha]$.
\end{enumerate}
If (i) and (ii) hold, then $\int_{[0,\infty)} \widetilde{\varphi}_k(x) \mu(dx) = {\partial^k \over \partial \sigma^k}\raisebox{-.25ex}{\big|}_{\sigma = {1 \over 2} - \alpha\!} \bigl[ \int_{[0,\infty)} \! \bm{W}_{\!\alpha,\sigma}(x) \mu(dx) \bigr]$.
\end{proposition}

\begin{proof}
The proof relies heavily on the Laplace representation \eqref{eq:bW_laplacerep} and is completely analogous to that of \cite[Theorem 4.11]{zeuner1989}.
\end{proof}

The martingale property of $*$-moment functions applied to $*$-Lévy processes is given below.

\begin{proposition} \label{prop:iw_levy_momentmarting}
Let $\{\varphi_k\}_{k=1,2}$ be a pair of $*$-moment functions. Let  $X = \{X_t\}_{t \geq 0}$ be a $*$-Lévy process. Then:
\begin{enumerate}[itemsep=0pt,topsep=4pt]
\item[\textbf{(a)}] If $\mathbb{E}[\varphi_1(X_t)]$ exists for all $t > 0$, then the process $\bigl\{\varphi_1(X_t) - \mathbb{E}[\varphi_1(X_t)]\bigr\}_{t \geq 0}$ is a martingale;

\item[\textbf{(b)}] If, in addition, $\mathbb{E}[\varphi_2(X_t)]$ exists for all $t > 0$, then the process
\[
\bigl\{\varphi_2(X_t) - 2\varphi_1(X_t) \mathbb{E}[\varphi_1(X_t)] - \mathbb{E}[\varphi_2(X_t)] + 2\mathbb{E}[\varphi_1(X_t)]^2 \bigr\}_{t \geq 0}
\]
is a martingale.
\end{enumerate}
In particular, if we let $Y$ be the index Whittaker diffusion (Example \ref{exam:iw_diffusion}) and let $\lambda_1, \lambda_2$ be as in Proposition \ref{prop:iw_moment_oderep}, then the processes $\{\varphi_1(Y_t) - \lambda_1 t\}_{t \geq 0}$ and $\{\varphi_2(Y_t) - 2\lambda_1 t \varphi_1(Y_t) - \lambda_2 t + \lambda_1^2 t^2\}_{t \geq 0}$ are martingales.
\end{proposition}

\begin{proof}
The proof of (a)--(b) is identical to that of \cite[Proposition 6.11]{zeuner1992}.

If $\{\mu_t\}$ is the $*$-Gaussian convolution semigroup associated with the index Whittaker diffusion, then $\widehat{\mu_t}(\lambda) = e^{-t\lambda}$; consequently (by Proposition \ref{prop:iw_moment_finitenesschar})
\[
\mathbb{E}[\widetilde{\varphi}_k(Y_t)] < \infty \: \text{ for all } \, t \geq 0, \qquad\quad \lim_{t \downarrow 0} \mathbb{E}[\widetilde{\varphi}_k(Y_t)] = 0 \qquad\quad (k = 1,2).
\]
It follows from Proposition \ref{prop:iw_moment_oderep} that any pair $\{\varphi_k\}_{k=1,2}$ can be written as a linear combination of $\widetilde{\varphi}_1$ and $\widetilde{\varphi}_2$, hence we also have $\lim_{t \downarrow 0} \mathbb{E}[\varphi_k(Y_t)] = 0$ for $k = 1,2$. We can now use (the proof of) \cite[Equations (4.5) and (4.6)]{rentzschvoit2000} to deduce that $\mathbb{E}[\varphi_1(Y_t)] = \lambda_1 t$ and $\mathbb{E}[\varphi_2(Y_t)] = \lambda_1^2 t^2 + \lambda_2 t$, and this concludes the proof.
\end{proof}

\subsection{Lévy-type characterization of the Shiryaev process} \label{sec:iw_shiryaev_levychar}

In the remainder of this paper we will show that the martingale property given in the last statement of the previous proposition is in fact a characterization of the index Whittaker diffusion and, therefore, yields a Lévy-type characterization for the Shiryaev process. For this purpose, it is convenient to focus on the moment functions $\phi_1$ and $\phi_2$ that correspond to the choice $\lambda_1 = -1$ and $\lambda_2 = 0$, i.e.
\begin{align*}
\phi_1(x) & = \int_0^x {1 \over y^{2\,} \mathrm{m}(y)} \! \int_0^y \mathrm{m}(\xi) \, d\xi\, dy \, = \, {1 \over 4(1-2\alpha)} \widetilde{\varphi}_1(x) \\
\phi_2(x) & = 2\int_0^x {1 \over y^{2\,} \mathrm{m}(y)} \! \int_0^y \mathrm{m}(\xi) \, \phi_1(\xi) \, d\xi\, dy = {1 \over 4(1-2\alpha)} \Bigl[\widetilde{\varphi}_2(x) - {2 \over 1-2\alpha} \widetilde{\varphi}_1(x)\Bigr].
\end{align*}

In the following results, we write
\[
\mathrm{C}^{k,\ell}[0,\infty) := \{f \in \mathrm{C}^k[0,\infty): f\raisebox{-.2ex}{\big|}_{[0,\eps)\!} \in \mathrm{C}^\ell[0,\eps) \text{ for some } \eps > 0\}
\]
and we denote by $[X] = \{[X]_t\}_{t \geq 0}$ the quadratic variation of a stochastic process $X$.

\begin{lemma} \label{lem:iw_moment_philemma}
\textbf{(a)} If $f \in \mathrm{C}^{2,4}[0,\infty)$ with $f'(0) = f'''(0) = 0$, then there exists $h \in \mathrm{C}^2[0,\infty)$ with $f(x) = h(\phi_1(x))$ for $x \geq 0$. \\[-10pt]

\textbf{(b)} There exists a unique function $h_0 \in  \mathrm{C}^2[0,\infty)$ such that $h_0(\phi_1(x)) = \phi_2(x)$, and it satisfies $h_0''(x) > 0$ for all $x \geq 0$.
\end{lemma}

\begin{proof}
From \eqref{eq:iw_moment_canontaylor} we find that the Taylor expansions of the moment functions $\phi_1$ and $\phi_2$ as $x \to 0$ are of the form $\phi_1(x) = c_1 x^2 + c_2 x^4 + o(x^4)$ and $\phi_2(x) = c_3 x^4 + o(x^4)$, with $c_1, c_3 > 0$. Consequently, the result is proved as in \cite[Lemmas 5.7 and 5.8]{rentzschvoit2000}, with the obvious adaptations.
\end{proof}

\begin{lemma} \label{lem:iw_moment_martlemma}
Let $X = \{X_t\}_{t \geq 0}$ be an $[0,\infty)$-valued process with a.s.\ continuous paths and such that the processes $Z\text{\large\mathstrut}_{X}^{-\mathcal{L},\phi_j}$ defined by \eqref{eq:iw_levy_Zdef} are local martingales for $j=1,2$. Then 
\[
\bigl[ Z\text{\large\mathstrut}_{X}^{-\mathcal{L},\phi_1} \bigr]_t = \frac{1}{2} \int_0^t X_s^2 (\phi_1'(X_s))^{2\,} ds \qquad \text{ almost surely}.
\]
Moreover, $Z\text{\large\mathstrut}_{X}^{-\mathcal{L},f}$ is a local martingale whenever $f \in \mathrm{C}^{2,4}[0,\infty)$ with $f'(0) = f'''(0) = 0$.
\end{lemma}

\begin{proof}
This proof is analogous to that of \cite[Lemma 6.2]{rentzschvoit2000}, to which we refer for further details.

Let $h \in \mathrm{C}^2[0,\infty)$. Given that $\mathcal{L} \phi_1 = -1$, an application of the chain rule shows that $\mathcal{L}(h(\phi_1))(x) = -{x^2 \over 4} h''(\phi_1(x))\, (\phi_1'(x))^2 - h'(\phi_1(x))$. Since by Itô's formula we have $d(h(\phi_1(X_t))) = h'(\phi_1(X_t))\, d\mskip 0.6\thinmuskip \phi_1(X_t)$ $+ {1 \over 2} h''(\phi_1(X_t))\, d[\phi_1(X)]_t$, we obtain
\begin{equation} \label{eq:iw_moment_localmartlemma_pf1}
d(h(\phi_1(X_t))) + \mathcal{L}(h(\phi_1(X_t)))\, dt = h''(\phi_1(X_t)) \Bigl(\tfrac{1}{2} d[\phi_1(X)]_t - \tfrac{1}{4}X_t^2(\phi_1'(X_t))^2 dt\Bigr) + dV_t^h
\end{equation}
where $\bigl\{V_t^h := \int_0^t h'(\phi_1(X_s)) (d\mskip 0.6\thinmuskip \phi_1(X_s) - ds)\bigr\}$ is a local martingale (because $d\mskip 0.6\thinmuskip \phi_1(X_t) - dt = d\mskip 0.4\thinmuskip Z\text{\large\mathstrut}_{X,t}^{-\mathcal{L},\phi_1}$ is the differential of a local martingale). If, in particular, $h$ is the function $h_0$ from Lemma \ref{lem:iw_moment_philemma}(b), then $\int_0^t d(h_0(\phi_1(X_s))) + \mathcal{L}(h_0(\phi_1(X_s)))\, ds = Z\text{\large\mathstrut}_{X,t}^{-\mathcal{L},\phi_2}$, and from \eqref{eq:iw_moment_localmartlemma_pf1} we find that
\[
\int_0^t  h_0''(\phi_1(X_s)) \Bigl(\tfrac{1}{2} d[\phi_1(X)]_s - \tfrac{1}{4}X_s^2 (\phi_1'(X_s))^2 ds\Bigr) \qquad \text{is a local martingale}.
\]
But, by standard results, $\int_0^t  h_0''(\phi_1(X_s)) \bigl(\tfrac{1}{2} d[\phi_1(X)]_s - \tfrac{1}{4}X_s^2 (\phi_1'(X_s))^2 ds\bigr)$ is also a process of locally finite variation, hence it is a.s.\ equal to zero. Consequently, taking into account that $h_0'' > 0$ (Lemma \ref{lem:iw_moment_philemma}(b)), we have $d[Z\text{\large\mathstrut}_{X}^{-\mathcal{L},\phi_1}]_t - \frac{1}{2} X_t^2 (\phi_1'(X_t))^2 dt = d[\phi_1(X)]_t - \frac{1}{2} X_t^2 (\phi_1'(X_t))^2 dt = 0$ a.s., proving the first assertion. 

The result just proved, combined with \eqref{eq:iw_moment_localmartlemma_pf1}, implies that $\bigl\{h(\phi_1(X_t)) + \int_0^t \mathcal{L}(h(\phi_1))(X_s) ds\bigr\}_{t \geq 0}$ is, for each $h \in \mathrm{C}^2[0,\infty)$, a local martingale. The second assertion thus follows from Lemma \ref{lem:iw_moment_philemma}(a).
\end{proof}

We are finally ready to establish the martingale characterization of the index Whittaker diffusion (which should be compared with \cite[Theorem 6.3]{rentzschvoit2000}):

\begin{theorem} \label{thm:iw_moment_diffusion_martchar}
Let $Y = \{Y_t\}_{t \geq 0}$ be a $[0,\infty)$-valued Markov process with a.s.\ continuous paths. The following assertions are equivalent:
\begin{enumerate}[itemsep=0pt,topsep=4pt]
\item[\textbf{(i)}] $Y$ is the index Whittaker diffusion;

\item[\textbf{(ii)}] $\{\phi_1(Y_t) - t\}_{t \geq 0}$ and $\{\phi_2(Y_t) - 2t \phi_1(Y_t) + t^2\}_{t \geq 0}$ are martingales (or local martingales);

\item[\textbf{(iii)}] $Z\text{\large\mathstrut}_{Y}^{-\mathcal{L},\phi_1}$ is a local martingale with $[Z\text{\large\mathstrut}_{Y}^{-\mathcal{L},\phi_1}]_t = \frac{1}{2} \int_0^t Y_s^2 (\phi_1'(Y_s))^{2\,} ds$.
\end{enumerate}
\end{theorem}

\begin{proof}
\textbf{(i) $\!\implies\!$ (ii):} This follows from Proposition \ref{prop:iw_levy_momentmarting}. \\[-10pt]

\textbf{(ii) $\!\implies\!$ (iii):} Assume that (ii) is true. Since $d\mskip 0.4\thinmuskip Z\text{\large\mathstrut}_{Y,t}^{-\mathcal{L},\phi_1} = d\phi_1(Y_t) - dt$, the process $Z\text{\large\mathstrut}_{Y}^{-\mathcal{L},\phi_1}$ is a local martingale. Furthermore,
\[
d\mskip 0.4\thinmuskip Z\text{\large\mathstrut}_{Y,t}^{-\mathcal{L},\phi_2} = \, d\phi_2(Y_t) - 2\phi_1(Y_t) dt \: = \: d\bigl(\phi_2(Y_t) - 2t\phi_1(Y_t) + t^2 \bigr) +  2t\bigl(d\phi_1(Y_t) - dt\bigr)
\]
(where integration by parts gives the second equality) and therefore the process $Z\text{\large\mathstrut}_{Y}^{-\mathcal{L},\phi_2}$ is also a local martingale. By Lemma \ref{lem:iw_moment_martlemma}, $[Z\text{\large\mathstrut}_{Y}^{-\mathcal{L},\phi_1}]_t = \frac{1}{2} \int_0^t Y_s^2 (\phi_1'(Y_s))^{2\,} ds$.\\[-10pt]

\textbf{(iii) $\!\implies\!$ (i):} Assuming that (iii) holds, Equation \eqref{eq:iw_moment_localmartlemma_pf1} and the proof of Lemma \ref{lem:iw_moment_martlemma} show that, for each $\lambda \geq 0$,  $Z\text{\large\mathstrut}_{X}^{-\mathcal{L},\raisebox{.5pt}{\scalebox{0.75}{\footnotesize$\bm{W}_{\!\alpha,\Delta_{\lambda\!}}(\cdot)$}}}$ is a local martingale and (by boundedness on compact time intervals) a true martingale. Proposition \ref{prop:iw_levy_equivchar} now yields that $Y$ is the index Whittaker diffusion.
\end{proof}

\begin{remark}
In this article we have focused on continuous-time stochastic processes which are additive with respect to the index Whittaker convolution. In a similar way, one can introduce the discrete-time counterparts of the processes studied above. 

A $[0,\infty)$-valued Markov chain $\{S_n\}_{n \in \mathbb{N}_0}$ with $S_0 = 0$ is said to be \emph{$*$-additive} if there exist measures $\mu_n \in \mathcal{P}[0,\infty)$ such that \begin{equation} \label{eq:iw_markovadditive_def}
P[S_n \in B | S_{n-1} = x] = (\mu_n * \delta_x)(B), \qquad\;\; n \in \mathbb{N}, \: x \geq 0, \: B \text{ a Borel subset of } [0,\infty).
\end{equation}
If $\mu_n = \mu$ for all $n$, then $\{S_n\}$ is said to be a \emph{$*$-random walk}. 

It is important to note that $*$-additive Markov chains can be constructed explicitly: letting $X_1$, $U_1$, $X_2$, $U_2$, $\ldots$ be a sequence of independent random variables (on a given probability space $(\Omega, \mathfrak{A}, P)$) where the $X_n$ have distribution $P_{X_n} = \mu_n$ and each of the (auxiliary) random variables $U_n$ has the uniform distribution on $[0,1)$, we obtain a $*$-additive Markov chain satisfying \eqref{eq:iw_markovadditive_def} by setting $S_0 = 0$ and $S_n = S_{n-1} \oplus_{\mathsmaller{U_n}} X_n$, where $X \oplus_\mathsmaller{U} Y := \max\bigl( 0, \sup\{ z \in [0,\infty): (\mathcal{T}^Y\mathds{1}_{[0,z]})(X) < U \} \bigr)$. The fact that  $X \oplus_\mathsmaller{U} Y$ is well-defined as a random variable and that its distribution is $P_{X \oplus_\mathsmaller{U} Y} = P_X * P_Y$ is justified as in \cite[Theorem 7.1.3 and Proposition 7.1.6]{bloomheyer1995}.

In the context of hypergroups, moment functions have been successfully applied to the study of the limiting behavior of additive Markov chains (cf.\ \cite[Chapter 7]{bloomheyer1995}). Parallel results hold for the index Whittaker convolution. For instance, letting $\{S_n\}$ be a $*$-additive Markov chain constructed as above, the following strong laws of large numbers are established as in \cite[Theorems 7.3.21 and 7.3.24]{bloomheyer1995}: \\[-9pt]

\emph{\textbf{(a)} If $\{r_n\}_{n \in \mathbb{N}}$ is a sequence of positive numbers such that $\lim_n r_n = \infty$ and $\sum_{n=1}^\infty {1 \over r_n} \bigl(\mathbb{E}[\widetilde{\varphi}_2(X_n)] - \mathbb{E}[\widetilde{\varphi}_1(X_n)]^2\bigr) < \infty$, then
\[
\lim_n {1 \over \sqrt{r_n}} \bigl( \widetilde{\varphi}_1(S_n) - \mathbb{E}[\widetilde{\varphi}_1(S_n)] \bigr) = 0 \qquad\;\; P\text{-a.s.}
\]}

\emph{\textbf{(b)} If $\{S_n\}$ is a $*$-random walk such that $\mathbb{E}[\widetilde{\varphi}_2(X_1)^{\theta/2}] < \infty$ for some $1 \leq \theta < 2$, then $\mathbb{E}[\widetilde{\varphi}_1(X_1)] < \infty$ and
\[
\lim_n {1 \over n^{1 / \theta}} \bigl(\widetilde{\varphi}_1(S_n) - n\mathbb{E}[\widetilde{\varphi}_1(X_1)]\bigr) = 0 \qquad\;\; P\text{-a.s.}
\]}
\end{remark}

\section*{Acknowledgments}

The first and third authors were partly supported by CMUP (UID/MAT/00144/2013), which is funded by Fundação para a Ciência e a Tecnologia (FCT) (Portugal) with national (MEC), European structural funds through the programmes FEDER under the partnership agreement PT2020, and Project STRIDE -- NORTE-01-0145-FEDER-000033, funded by ERDF -- NORTE 2020. The first author was also supported by the grant PD/BD/135281/2017, under the FCT PhD Programme UC|UP MATH PhD Program. The second author was partly supported by FCT/MEC through the project CEMAPRE -- UID/MULTI/00491/2013.


\end{document}